\documentclass[a4paper, 11pt, parskip=half]{amsart}

\usepackage{custom-template, shortcuts}

\begin{document}
	\title{On the Lang--Trotter conjecture for Siegel modular forms}
	\date{}
	\author{Arvind Kumar, Moni Kumari and Ariel Weiss}
	\address{Arvind Kumar, Department of Mathematics, Indian Institute of Technology Jammu, Jagti, PO Nagrota, NH-44 Jammu 181221, J \& K, India\vspace*{-5pt}}
	\email{arvind.kumar@iitjammu.ac.in\vspace*{-6pt}}
	\address{Moni Kumari, Department of Mathematics, Indian Institute of Technology Jammu, Jagti, PO Nagrota, NH-44 Jammu 181221, J \& K, India\vspace*{-5pt}}
	\email{moni.kumari@iitjammu.ac.in\vspace*{-6pt}}
	\address{Ariel Weiss, Department of Mathematics, The Ohio State University, Columbus, Ohio, USA\\Department of Mathematics, Ben-Gurion University of the Negev, Be'er Sheva 8410501, Israel.\vspace*{-5pt}}
    \email{weiss.742@osu.edu\vspace*{-6pt}}

	
	\begin{abstract}
        Let $f$ be a genus two cuspidal Siegel eigenform. We prove an adelic open image theorem for the compatible system of Galois representations associated to $f$, generalising the results of Ribet and Momose for elliptic modular forms. Using this result, we investigate the distribution of the Hecke eigenvalues $a_p$ of $f$, and obtain upper bounds for the sizes of the sets $\{p \le x : a_p = a\}$ for fixed $a\iC$, in the spirit of the Lang--Trotter conjecture for elliptic curves.
	\end{abstract}
	\maketitle

\vspace{-.6cm}
    
	\section{Introduction}
	If $A$ is a non-CM elliptic curve over $\Q$ of conductor $N$ and if $a\iZ$, then the \emph{Lang--Trotter conjecture} \cite{lang-trotter} states that
		\begin{equation*}\label{eq:lt-conjecture}
\pi_A(x, a):= \#\set{p\le x,\ p\nmid N : a_p = a} \sim C(A, a)\frac{ x^{1/2}}{\log x},
		\end{equation*}
where $a_p = p+1 - \#A(\Fp)$ and $C(A, a)\ge 0$ is  an explicit constant.
The conjecture is formulated in terms of the two-dimensional compatible system of Galois representations attached to $A$. More generally, if $(\rho_\l)_\l$ is an arbitrary compatible system of Galois representations of conductor $N$, whose Frobenius polynomials are defined over a number field $E$, then for each $a\in\O_E$, it is natural to ask for the asymptotics of the size of the set
	\begin{equation}\label{eq:lt-set}
	    \set{p\le x,\ p\nmid \l N : \tr\rho_\l(\Frob_p) = a}
	\end{equation}
	where $\Frob_p$ denotes a Frobenius element of $\Gal(\Qpb/\Qp)$. 
When $(\rho_\l)_\l$ is self-dual up to twist, a generalisation of the Lang--Trotter conjecture has been formulated by V.\ K.\ Murty \cite{murty-Frobenius}*{Conj.\ 2.15}, under mild assumptions.   
	 
	The goal of this paper is to estimate the size of the set $(\ref{eq:lt-set})$ when $(\rho_\l)_\l$ is the compatible system of Galois representations attached to a genus two Siegel modular eigenform. 
	
	Let $f$ be a cuspidal vector-valued genus two Siegel modular eigenform of weight $(k_1, k_2)$---i.e.\ weight $\Sym^{k_1-k_2}\det^{k_2}$---with $k_1\ge k_2\ge 2$ and level $N$. Let $\pi$ be the cuspidal automorphic representation of $\Gf(\A_\Q)$ associated to $f$, and assume that its functorial lift to $\GL_4(\A_\Q)$ exists, is cuspidal, and is neither an automorphic induction nor a symmetric cube lift. Let $\varepsilon$ be the central character of $\pi$. Let $E = \Q(\{a_p : p\nmid N\}, \varepsilon)$ be the number field generated by the image of $\varepsilon$ and by the eigenvalues $a_p$ of the Hecke operators $T_p$. Let $F\sub E$ be the field fixed by the inner twists of $f$ (see \Cref{def: twists-intro} and \Cref{sec:inner-twists}).  
	
	\begin{theorem}\label{thm:unconditional-lta} 
	    For $a\in \O_F$, let
	    \[\pi_f(x, a) := \#\set{p\le x : a_p = a}.\]
	    Then, for any $\epsilon>0$,
	    \[\pi_f(x, a) \ll_{\epsilon,f}\frac{x}{(\log x)^{1+\alpha-\epsilon}},\quad\alpha = \frac{[F:\Q]}{10[F:\Q]+1}.\]
	    If $a = 0$, then
	     \[\pi_f(x, 0) \ll_{\epsilon,f} \frac{x}{(\log x)^{1+\alpha-\epsilon}},\quad\alpha = \frac{[F:\Q]}{7[F:\Q]+1}.\]
	\end{theorem}
	
	If we further assume the Generalised Riemann Hypothesis (GRH), we obtain the following strengthened result:
	
	\begin{theorem}\label{thm:conditional-lta}
		Let $f$ be as above, and assume GRH. If $a\in \O_F$, then
		\[\pi_f(x, a) \ll_{f}\frac{x^{1-\alpha}}{(\log x)^{1-2\alpha}},\quad\alpha= \frac{[F:\Q]}{11[F:\Q]+1}.\]
		If $a = 0$, then
	     \[\pi_f(x, 0) \ll_{f}\frac{x^{1-\alpha}}{(\log x)^{1-2\alpha}},\quad\alpha= \frac{[F:\Q]}{10[F:\Q]+1}.\]
	\end{theorem}
	
\begin{remark}\label{rem:arthur}
 The assumption that the functorial lift of $\pi$ is cuspidal and not an automorphic induction or a symmetric cube lift is equivalent to demanding that $f$ is of type \textbf{(G)} in the notation of \cite{schmidt}, does not have real multiplication (RM) or complex multiplication (CM), and is not a symmetric cube lift. In these cases, stronger versions of \Cref{thm:unconditional-lta} and \ref{thm:conditional-lta} follow from the analogous results for modular forms.
 
 The existence of this functorial lift follows from work of Weissauer \cite{Weissauersymplectic} and Asgari--Shahidi \cite{asgari-shahidi} when $k_2 >2$. When $k_2 = 2$, the existence of this lift follows from Arthur's classification \cite{Arthur2013, geetaibi, arthur-plug}. However, this result is still not truly unconditional, and relies on several unpublished results that have yet to appear. Hence, we impose the existence of this lift as an assumption. We refer the reader to \cite{weiss2018image}*{Sec.\ 2.4} and \cite{calegari-blog} for further discussion. 
\end{remark}	
   
	The Lang--Trotter conjecture for elliptic curves was first investigated by Serre \cite{ser81}, who showed that $\pi_A(x, a) \ll \frac{x}{(\log x)^{5/4-\epsilon}}$ unconditionally, and that $\pi_A(x, a) \ll x^{7/8}(\log x)^{1/2}$ under GRH. These bounds were subsequently improved by Wan \cite{Wan} and V.\ K.\ Murty \cite{murty-mod-forms-ii} unconditionally, and by Murty--Murty--Saradha under GRH \cite{mms}. The best current estimates when $a\ne 0$ are $\pi_A(x, a) \ll \frac{x(\log\log x)^2}{(\log x)^{2}}$ unconditionally \cite{thorner-zaman}, and $\pi_A(x, a) \ll \frac{x^{4/5}}{(\log x)^{3/5}}$ under GRH \cite{zyw}.
	
	These results all apply verbatim to non-CM elliptic modular forms of weight $\ge 2$ with integer Hecke eigenvalues, and can easily be adapted to give bounds when the field $E$ generated by the Hecke eigenvalues is arbitrary (see, for example, \cite{ser81}*{Sec.\ 7}). The bounds obtained in this way are independent of the field $E$ and its subfield $F$. A novel feature of our result is that our bounds improve as the degree $[F:\Q]$ increases.
	
	The higher dimensional case has been studied by Cojocaru--Davis--Silverberg--Stange \cite{cojocaru-abelian-varieties} (see also \cite{serban, cojocaru-wang, cojocaru2022bounds}), who formulate a precise conjecture for generic abelian varieties, and prove analogues of Theorems \ref{thm:unconditional-lta} and \ref{thm:conditional-lta}. If $A$ is an abelian surface over $\Q$ with $\End(A)=\Z$, then, conjecturally, for each prime $\l$, the $\l$-adic Galois representation attached to $A$ should be isomorphic to the $\l$-adic Galois representation attached to a Siegel modular form $f$ of weight $(2,2)$, paramodular level and integer Hecke eigenvalues. For such a Siegel modular form $f$, our unconditional bound in \Cref{thm:unconditional-lta} exactly matches that of \cite{cojocaru-abelian-varieties}*{Thm.\ 1} (see also \Cref{rem:strengthened-result}). Our conditional bound of $O(x^{11/12}(\log x)^{-5/6})$ in \Cref{thm:conditional-lta} agrees with that of \cite{cojocaru2022bounds} when $g =2$ and $F=\Q$, and is slightly stronger than the bound $O(x^{21/22+\epsilon})$ of \cite{cojocaru-abelian-varieties}. The authors of \cite{cojocaru-abelian-varieties} also formulate a precise conjecture for the asymptotics of $\pi_A(x, a)$, including the constant $C(A, a)$. It would be interesting to formulate such a conjecture in the case of Siegel modular forms, particularly when $F\ne \Q$, however, we do not pursue that here.
	
	The proofs of the above results all use the strategy initiated by Serre \cite{ser81}, which combines explicit versions of the Chebotarev density theorem with precise calculations of the images of Galois representations attached to modular forms and abelian varieties \cite{serre-abelian, Ribet75, momose, Ribet85}. While our approach is inspired by this framework, a key obstacle in the case of Siegel modular forms is that these image results are not available. The key technical input of this paper is a precise big image theorem for Galois representations attached to Siegel modular forms.

	\subsection{Images of Galois representations}
	
	There is a general philosophy that the image of an automorphic Galois representation should be as large as possible, unless there is an automorphic reason for it to be small. For example, let $f$ be a Siegel modular eigenform as in the previous section. For each prime $\lambda$ of $E$, there is a $\lambda$-adic Galois representation
	\[\rho_{\lambda}\:\Ga\Q\to \Gf(\overline{E}_\lambda)\]
	associated to $f$. If $f$ is not of type \textbf{(G)} in the notation of \cite{schmidt}, then $\rho_{\lambda}$ is reducible for all $\lambda$. Similarly, the image of $\rho_{\lambda}$ is small if $f$ has complex or real multiplication, or is a symmetric cube lift. 
	
	In \cite{weiss2018image}, building on previous work of Dieulefait and Dieulefait--Zenteno \cites{Dieulefait2002maximalimages, DZ}, the third author showed that if $f$ is not in one of these exceptional cases, then the image of $\rho_\lambda$ is large in the following sense: for all but finitely many primes $\lambda$ (or for all $\lambda\mid\l$ for a set of primes $\l$ of density $1$ if $k_2 = 2$), the image of the residual representation $\orho_{\lambda}\:\Ga\Q\to \Gf(\O_E/\lambda)$ contains $\Sp_4(\Fl)$. 
	
	However, in order to prove Theorems \ref{thm:unconditional-lta} and \ref{thm:conditional-lta}, we need to pin down the exact image $\orho_{\lambda}$ for almost all primes. This study is complicated by an additional symmetry, that of \emph{inner twists}, first described by Ribet \cite{Ribet77,Ribet-twists} for elliptic modular forms. Fix an embedding $E\hookrightarrow\C$.
	
	\begin{definition}\label{def: twists-intro}
	    An \emph{inner twist} of $f$ is a pair $(\sigma, \chi_\sigma)$, where $\sigma\in\Hom(E, \C)$ and $\chi_\sigma$ is a Dirichlet character, such that $\sigma(a_p) = \chi_\sigma(p)a_p$ for almost all primes $p$.
	\end{definition}
	
	We show in \Cref{sec:inner-twists} that the set of such $\sigma\in \Hom(E,\C)$ forms an abelian subgroup $\Gamma$ of $\Aut(E/\Q)$. Let $F = E^\Gamma$ be its corresponding fixed field. For each $\sigma\in\Gamma$, let $K_\sigma$ be the number field cut out by $\chi_\sigma$, and let $K$ be the compositum of all the $K_\sigma$'s. If $p$ splits in $K$, we have $\sigma(a_p) = a_p$ for all $\sigma\in\Gal(E/F)$, i.e.\ $a_p\in F\subsetneq E$. Thus, inner twists give a restriction on the image of $\rho_{\lambda}$: the Frobenius elements associated to a positive density of primes $p$ have trace contained in the proper subfield $F$ of $E$. Let $\G=\G_f$ be the group scheme over $\Z$ such that, for each $\Z$-algebra $R$, we have
	\begin{equation*}
	\G(R) = \set{(g, \nu)\in \Gf(\O_F\tensor_\Z R)\times R\t : \simil(g) = \nu^{k_1 + k_2 - 3}}.    
	\end{equation*}
		Here, $\simil\:\Gf\to \GL_1$ is the similitude character. Let $\rho_{\l}:= \bigoplus_{\lambda\mid\l}\rho_{\lambda}$. We show in \Cref{lem:small-trace} that, for almost all primes $\l$, the restriction
	\[\rho_{\l}|_K\:\Gal(\Qb/K)\to \Gf(E\tensor_\Q\Qlb)\]
	factors through $\Gf(F\tensor_\Q\Ql)$, and can be viewed as a representation
	\[\wrho_{ \l}\:\Gal(\Qb/K)\to \G(\Ql)\]
	such that the projection to $\Gf(\O_F\tensor_\Z\Ql)$ is $\rho_{\l}|_K$ and the projection to $\Ql\t$ is the cyclotomic character. Moreover, up to conjugation, we can assume that $\wrho_{\l}$ takes values in $\G(\Zl)$. If $\LL$ is a set of rational primes, let $\widehat{\Q}_\LL = \Q\tensor_\Z\prod_{\l\in\LL}\Zl$ and let
	\[\wrho_{ \LL}:= \bigoplus_{\l\in\LL}\wrho_{ \l}\:\Gal(\Qb/K)\to \G(\widehat\Q_\LL)\]
	be the associated adelic Galois representation.
	
	Our main technical result is the following determination of the images of these Galois representations, which generalises the results of Serre for elliptic curves \cite{serre-abelian}
    and of Ribet, Momose and Loeffler for elliptic modular forms \cites{Ribet75, Ribet77,momose,Ribet85, Loeffler-adelic}.
    
	\begin{theorem}\label{thm:precise-image}
	    Let $f$ be a cuspidal vector-valued Siegel modular eigenform of weight $(k_1, k_2)$ and level $N$. Define $E, F$ and $\G$ as above. Let $\pi$ be the cuspidal automorphic representation of $\Gf(\A_\Q)$ associated to $f$ and assume that its functorial lift to $\GL_4(\A_\Q)$ exists, is cuspidal, and is neither an automorphic induction nor a symmetric cube lift. Let $\varepsilon$ be its central character.
	    
	    Let $\LL'$ be the set of rational primes $\l$ such that $\l\ge 5$, such that $\rho_{\l}|_{\Ql}$ is de Rham and such that $\rho_{\l}|_{\Ql}$ is crystalline if $\l\nmid N$.\footnote{If $k_2>2$, then $\LL'$ consists of all primes $\l\ge 5$. If $k_2 = 2$, then $\LL'$ has Dirichlet density $1$ by \cite{weiss2018image}*{Thm.\ 1.1}.} Let $\LL\sub \LL'$ be the $($cofinite$)$ subset of primes such that $\rho_{\l}|_K$ takes values in $\Gf(F\tensor_\Q\Ql)$. Then:
	
			\begin{enumerate}
			\item For each prime $\l\in\LL$, the image of $\wrho_{\l}$ is an open subgroup of $\G(\Zl)$.
			\item For all but finitely many primes $\l\in\LL$, the image of $\wrho_{\l}$ is exactly $\G(\Zl)$.
			\item The image of $\wrho_{\LL}$ is an open subgroup of $\G(\widehat{\Q}_\LL)$.
		\end{enumerate}
	\end{theorem}

\subsection{Methods}
    \subsubsection{The image of Galois}
    
    To prove \Cref{thm:precise-image}, in \Cref{sec:inner-twists}, we first generalise the notion of inner twists to the setting of Siegel modular forms, and show that they give a restriction on the image of $\rho_\l$. A key technical input is \Cref{lem:b_q}, which shows that the field $F$ cut out by the inner twists of $f$ is the trace field of the standard representation of $\rho_\l$, obtained by composing $\rho_\l$ with the maps $\Gf\to\PGSp_4\xrightarrow{\sim}\SO_5\to\GL_5$. As a result, we deduce that $F$ is generated over $\Q$ by $\{\frac{b_p}{\varepsilon(p)} : p\nmid N \}$, where $b_p$ is the coefficient of $X^2$ in the characteristic polynomial of $\rho_\l(\Frob_p)$, which does not depend on $\l$.
    
    Fix a prime $\l\in\LL$. To prove that $\wrho_\l$ has open image, using the definition of the inner twists and an argument using Goursat's lemma, we prove that it is enough to show that $\rho_\lambda|_K$ has open image inside
    \[G_\lambda := \set{g\in\Gf(\O_{F_\lambda}): \simil(g) \in\Zl^{\times(k_1 + k_2 -3)}}\]
    for each $\lambda\mid\l$. Let $H_\lambda$ be the image of $\rho_\lambda|_K$. Our starting point for proving that $\rho_\lambda|_K$ has open image is the result of \cite{weiss2018image}, that $H_\lambda$ is a Zariski dense subgroup of $\Gf(F_\lambda)$. In \Cref{sec:open-image}, we use this result in combination with a theorem of Pink \cite{Pink} to show that the projective image of $\rho_\lambda|_K$ is open in $\PGSp_4(F_\lambda)$, from which we can deduce the result.
    
    Similarly, to prove that $\wrho_\l$ is surjective, again using Goursat's lemma, we prove that it is enough to show that $\rho_\lambda$ surjects onto $G_\lambda$ for all $\lambda\mid\l$. The surjectivity of $\rho_\lambda$ was proven by Dieulefait \cite{Dieulefait2002maximalimages}*{Sec.\ 4.7} under the assumption that the field $F=\Q(\{\frac{b_p}{\varepsilon(p)} : p\nmid N \})$ is generated over $\Q$ by $b_p$ for a single prime $p$. In \Cref{lem:untwisted}, using our proof that the image of $\wrho_\l$ is open in $\G(\Zl)$, we prove that Dieulefait's assumption holds unconditionally.
    
    Once we have proven parts $(i)$ and $(ii)$ of \Cref{thm:precise-image}, part $(iii)$ follows from a straightforward but technical generalisation of the group-theoretic results of Serre \cite{serre-abelian}, Ribet \cite{Ribet75} and Loeffler \cite{Loeffler-adelic}.
    
    \subsubsection{Lang--Trotter bounds}
    
    Our proofs of \Cref{thm:unconditional-lta} and \Cref{thm:conditional-lta} are very different in nature. To prove \Cref{thm:unconditional-lta},
    we use the machinery of Serre \cite{ser81}, which works by combining the explicit Chebotarev density theorem of \cite{lo} with the $\l$-adic image of Galois for a single prime $\l$. Rather than applying this machinery to the $\lambda$-adic Galois representation $\rho_\lambda$, as Serre does for elliptic modular forms in \cite{ser81}*{Sec.\ 7}, the precision of \Cref{thm:precise-image} allows us to apply Serre's machinery to the full $\l$-adic Galois representation $\rho_\l=\bigoplus_{\lambda\mid\l}\rho_\lambda$. As a result, we obtain stronger bounds that improve with the size of $[F:\Q]$. For example, applying our method in the case of an elliptic modular form $f$ improves Serre's bound of $\pi_f(x, a)\ll_{f,\epsilon} x/(\log x)^{5/4-\epsilon}$ to the bound
    \[\pi_f(x, a) \ll_{\epsilon,f}\frac{x}{(\log x)^{1+\alpha-\epsilon}},\quad\alpha = \frac{[F:\Q]}{3[F:\Q]+1}.\]

Note that our bounds improve with the size of $[F:\Q]$. However, they remain far from the generalised Lang--Trotter conjecture of Murty \cite{murty-Frobenius}*{Conj.\ 3.1}, which predicts that as soon as $[F:\Q] \ge 3$, we should actually have $\pi_f(x, a) = O(1)$.
    
    In contrast, our proof of \Cref{thm:conditional-lta} generalises the methods of Murty--Murty--Saradha \cite{mms}, which work by combining the explicit Chebotarev density theorem with the mod $\l$ image of Galois for infinitely many primes $\l$. 
    
    Assume, for the sake of exposition, that $E = F$, so that $f$ has no inner twists, and $K=\Q$. Then, if $a\in\O_F$, by the definition of the residual representation $\worho_\l$, we have
    \begin{equation*}
        \pi_f(x, a) \le \#\set{p\le x : p\nmid \l N,\ \tr\worho_\l(\Frob_p)\equiv a\pmod\l} + O(1),    
    \end{equation*}
    where, by $\tr \worho_\l$, we mean the trace of the $\Gf(\O_F\tensor_\Z\Fl)$ component of $\G(\Fl)$, which takes values in $\O_F\tensor_\Z\Fl$.
        By \Cref{thm:precise-image}, the image of the residual representation $\worho_\l$ is $\G(\Fl)$ for all but finitely many $\l\in\LL$. Hence, for such $\l$, $\worho_\l$ factors through a finite Galois extension $L/K$, with Galois group $\G(\Fl)$. Thus, it is sufficient to bound the number of primes $p\le x$ such that $\worho_\l(\Frob_p)$ is contained in the conjugation invariant subset $\set{(g, \nu)\in \G(\Fl) : \tr(g)\equiv a\pmod\l}\sub\G(\Fl)$, which we can do using the Chebotarev density theorem. However, applying Chebotarev directly to these sets would not give the strongest possible bound.
        
    The key idea is that, under GRH, we can bound the size of $\pi_f(x, a)$ by bounding the size of the smaller set $\{p\le x : a_p = a,\ \l \text{ splits completely in } F(p)\}$ for almost all primes $\l$ that split completely in $F$ (\Cref{lem:murty-bound}). Here, $F(p)$ is the splitting field over $\Q$ of the characteristic polynomial of $\rho_\l(\Frob_p)$. If $p$ is in this smaller set, then the eigenvalues of $\worho_\l(\Frob_p)$ are in $(\O_F\tensor_\Z\Fl)\t$, so $\worho_\l(\Frob_p)$ is conjugate over $\Fl$ to an upper triangular matrix. We show in \Cref{lem4} that we can bound this set by applying the Chebotarev density theorem to an abelian extension, whose Galois group is isomorphic to the group of upper triangular matrices in $\G(\Fl)$ modulo the subgroup of unipotent upper triangular matrices. In particular, since Artin's holomorphy conjecture is known for abelian extensions, we can bound this set by using a stronger version of the Chebotarev density theorem due to Zywina \cite{zyw}.
    
    As in the proof of \Cref{thm:unconditional-lta}, rather than working with the mod $\lambda$ Galois representations $\orho_\lambda\:\Ga\Q\to\Gf(\O_E/\lambda)$, using \Cref{thm:precise-image}, we work with the mod $\l$ Galois representations $\orho_\l\:\Ga\Q \to \Gf(\O_E\tensor_\Z\Fl)$. Applying our method in the case of an elliptic modular form $f$ gives the bound (under GRH)
    	\[\pi_f(x, a) \ll_{f}\frac{x^{1-\alpha}}{(\log x)^{1-2\alpha}},\quad\alpha=\frac{[F:\Q]}{4[F:\Q]+1}.\]
 
	\subsection{Outline of the paper}
	In \Cref{sec:galois}, we recall key properties of the Galois representations attached to Siegel modular eigenforms, and generalise the notion of inner twists to Siegel modular forms. In \Cref{sec:image}, we prove \Cref{thm:precise-image}, which is a key technical result. In \Cref{sec:cdt}, we recall explicit versions of the Chebotarev density theorem, a variant due to Serre \cite{ser81} and a refinement due to Zywina \cite{zyw}. Using these inputs, we prove Theorems \ref{thm:unconditional-lta} and  \ref{thm:conditional-lta} in Sections \ref{sec:proof-unconditional} and \ref{sec:proof-conditional}.   

\section{Galois representations attached to Siegel modular forms}\label{sec:galois}
	
	Let $f$ be a cuspidal Siegel modular eigenform of weights $(k_1, k_2)$, $k_1\ge k_2\ge 2$ and level $N$. We assume throughout this paper that if $\pi$ is the cuspidal automorphic representation of $\Gf(\AQ)$ attached to $f$, then the functorial lift of $\pi$ to $\GL_4(\A_\Q)$ exists, is cuspidal, and is neither an automorphic induction nor a symmetric cube lift (see \Cref{rem:arthur}). In particular, $f$ is of type \textbf{(G)} in the notation of \cite{schmidt}, and is not CM, RM or a symmetric cube lift. Let $\varepsilon$ be the central character of $\pi$ and let $E=\Q(\{a_p : p\nmid N\}, \varepsilon)$ be the subfield of $\C$ generated by the image of $\varepsilon$ and by the Hecke eigenvalues $a_p$ of the Hecke operators $T_p$. Then $E$ is a finite extension of $\Q$.
	
	For a ring $R$, let
        \[\Gf(R) = \set{g\in\GL_4(R): g^t J g = \nu J,\ \nu \in R\t},\]
        where  $J = \br{\begin{smallmatrix}
        	0 & 0 & 0 & 1\\
        	0 & 0 & 1 & 0\\
        	0 & -1 & 0 & 0\\
        	-1 & 0 & 0 & 0
        \end{smallmatrix}}$. This choice of $J$ ensures that the upper triangular matrices in $\Gf(R)$ are a Borel subgroup. For $g\in \Gf(R)$, the constant $\nu$ is called the \emph{similitude} of $g$ and is denoted $\simil(g)$. Let $\Sp_4(R)$ be the subgroup of elements for which $\simil(g) = 1$, let $\PSp_4(R) = \Sp_4(R)/Z(\Sp_4(R))$, and let $\PGSp_4(R) = \Gf(R)/Z(\GSp_4(R))$, where $Z(\Sp_4(R))$ and $Z(\Gf(R))$ are the centres of $\Sp_4(R)$ and $\Gf(R)$.  We define the trace function
        \[\tr\:\Gf(R) \to R\]
        by viewing $g\in\Gf(R)$ as an element of $\GL_4$. We similarly define the characteristic polynomial of $g$ as a polynomial over $R$.
        
	By the work of Taylor, Laumon and Weissauer \cites{taylor1993, Laumon, Weissauer, Weissauersymplectic} when $k_2\ge 2$, and Taylor \cite{taylor1991galois} when $k_2=2$ (see also \cite{mok2014galois}), for each prime $\lambda$ of $E$, there exists a continuous semisimple symplectic Galois representation
	\[\rho_\lambda\:\Ga\Q\to \Gf(\overline E_{\lambda})\]
	that is unramified at all primes $p\nmid \l N$, and is characterised by the property
	\[\Tr\rho_{\lambda}(\Frob_p)= a_p,\qquad \simil\rho_{\lambda}(\Frob_p) = \varepsilon(p)p^{k_1 + k_2 -3},\]
	for all primes $p\nmid \l N$. By \cite{serre-mod-p-lattices}*{Thm.\ 5.2.1}, we can view $\rho_\lambda$ as a representation valued in $\Gf(\overline \O_{E_\lambda})$, and define the mod $\lambda$ representation
	\[\orho_\lambda\:\Ga\Q\to \Gf( \O_E/\lambda)\]
	to be the semisimplification of the reduction of $\rho_\lambda$ mod $\lambda$. This reduction is still symplectic by \cite{voight-6author}*{Lemma 4.3.6}. The representation $\rho_\lambda$ should be isomorphic to a representation taking values in $\Gf(E_\lambda)$ rather than $\Gf(\elb)$, but this seems not to be known in general. This ambiguity does not occur for $\orho_\lambda$, since  mod $\l$ representations are always defined over their trace field. 
	
	By work of Ramakrishnan \cite{Ramakrishnan}, Dieulefait--Zenteno \cite{DZ} and the third author \cite{weiss2018image}, the image of $\rho_\lambda$ is generically large, in the following sense. Let $\LL'$ be the set of rational primes $\l$ such that $\l\ge 5$, such that $\rho_\l|_{\Ql}$ is de Rham and such that $\rho_\l|_{\Ql}$ is crystalline if $\l\nmid N$. Then $\LL'$ is just the set of primes $\l\ge 5$ if $k_2>2$, while if $k_2 = 2$, $\LL'$ has Dirichlet density $1$ \cite{weiss2018image}*{Thm.\ 1.1}.
	
	\begin{theorem}[\cite{weiss2018image}*{Thms.\ 1.1, 1.2}]\label{thm:image-contains-sp4}
		\begin{enumerate}[leftmargin=*]
			\item If $\l\in \LL'$ and $\lambda\mid\l$, then $\rho_\lambda$ is absolutely irreducible.
			\item For all but finitely many $\l\in\LL'$, if $\lambda\mid \l$, then the image of $\orho_\lambda$ contains a subgroup conjugate to $\Sp_4(\Fl)$.
		\end{enumerate}
	\end{theorem}

	Conjecturally, \Cref{thm:image-contains-sp4} should be true with $\LL'$ the set of all primes, however, this question is open.

	\begin{corollary}\label{cor:field-of-def}
		For all but finitely many primes $\l\in \LL'$, for each prime $\lambda\mid\l$, the Galois representation $\rho_\lambda$ descends to a representation
		\[\rho_\lambda\:\Ga\Q\to \Gf(E_\lambda).\]
	\end{corollary}
	
	\begin{proof}
		By \Cref{thm:image-contains-sp4}, $\orho_\lambda$ is irreducible for all $\lambda\mid\l$, for all but finitely many $\l\in\LL'$. Hence, for such primes $\lambda$, by \cite{voight-6author}*{Lemma 4.3.8} and \cite{carayol1994formes}*{Th\'eor\`eme 2}, $\rho_\lambda$ is defined over its trace field.
	\end{proof}
		
	For each prime $\l$, we form the $\l$-adic representation
	\[\rho_\l :=\bigoplus_{\lambda\mid\l}\rho_\lambda\:\Ga\Q\to \Gf( E\tensor_\Q\Qlb),\]
	which, as before, we may view as taking values in $\Gf(\O_E\tensor_\Z\Zlb)$, and the mod $\l$ representation
	\[\orho_\l :=\bigoplus_{\lambda\mid\l}\orho_\lambda\:\Ga\Q\to \Gf(\O_E\tensor_\Z\Fl).\]    
	Finally, let $\LL\sub\LL'$ be a set of rational primes such that, for each $\l\in \LL$, $\rho_\l$ takes values in $\Gf(E\tensor_\Q\Ql)$. If $\widehat\Q_\LL = \Q\tensor_{\Z}\prod_{\l\in\LL}\Zl$, then we define \[\rho_\LL:=\bigoplus_{\l\in\LL}\rho_\l\:\Ga\Q\to \Gf(E\tensor\widehat\Q_\LL).\]

	\subsection{Inner twists}\label{sec:inner-twists}
	
	In this section, we define the notion of \emph{inner twists} for Siegel modular forms and discuss their key properties. The results in this section are well known in the case of elliptic modular forms \cite{momose, Ribet85}, and the definitions have been generalised to Siegel modular forms (and to general symplectic representations) in \cite{de-reyna-dieulefait-wiese-1}. Fix once and for all an embedding $\sigma_0\:E\hookrightarrow \C$. In particular, via this embedding, we may view the eigenvalues $a_p$ as elements of $\C$.
	
	\begin{definition}
		An \emph{inner twist} of $f$ is a pair $(\sigma, \chi)$, where $\sigma\in \Hom(E, \C)$ and $\chi$ is a Dirichlet character, such that
		\[\sigma(a_p)= \chi(p)a_p\]
		for all primes $p\nmid N$.
	\end{definition}
	
	We define $\Gamma$ to be the set of $\sigma\in \Hom(E, \C)$ for which such a $\chi$ exists. It is simple to show that $(\sigma_0, \chi)$ is an inner twist for a non-trivial character $\chi$ if and only if $f$ is CM or RM. In particular, since we have assumed that $f$ is not CM or RM, for each twist $(\sigma, \chi)$, the character $\chi$ is uniquely determined by $\sigma$, and we denote it by $\chi_\sigma$.
	
	Fix an isomorphism $\C\cong\Qlb$ for each $\l$. Then each $\sigma\in\Hom(E,\C)$ induces a map $E\tensor\Ql\to\Qlb$. If the Galois representation $\rho_\l$ takes values in $\Gf(E\tensor\Ql)$, then, for each $\sigma\in\Hom(E, \C)$, we can define $\sig\rho_\l$ to be the composition of $\rho_\l$ with the map $\sigma\:\Gf(E\tensor\Ql)\to \Gf(\Qlb)$. By the Chebotarev density theorem, a pair $(\sigma, \chi)$ is an inner twist if and only if $\sig\rho_\l \simeq \prepower{\sigma_0}\rho_\l\otimes\chi$, where we view $\chi$ as a Galois character.
	
	\begin{proposition}[\cite{Ribet-twists}*{Prop.\ 3.2}]
		If $\sigma\in\Gamma$, then $\sigma(E)\sub E$.
	\end{proposition}
	
	\begin{proof}
		Let $(\sigma, \chi)$ be an inner twist. Comparing the similitudes of $\sig\rho_\l$ and $\prepower{\sigma_0}\rho_\l\otimes\chi$, we see that $\chi^2 =  \sigma(\varepsilon)\cdot\varepsilon\ii$. Thus $\chi$ takes values in the field $\Q(\varepsilon)\sub E$. But then
		\[\sigma(a_p) = \chi(p) a_p \in E\]
		for all $p\nmid\l N$. Applying the same argument with a different prime $\l$, it follows that $\sigma(a_p)\in E$ for all $p\nmid N$, and hence that $\sigma(E)\sub E$.
	\end{proof}
	
	In particular, $\Gamma$ is a subset of $\Aut(E/\Q)$. Moreover, the inner twists $(\sigma, \chi_\sigma)$ form a group under the multiplication
	\[(\sigma, \chi_\sigma)\cdot(\tau, \chi_\tau) = (\sigma\tau, \chi_\sigma\cdot \sigma(\chi_\tau)).\]
	Thus, $\Gamma$ is a subgroup of $\Aut(E/\Q)$.
	
	\begin{definition}\label{def:F}
	    Define $F = E^\Gamma$ to be the fixed field of $\Gamma$.
	\end{definition}
	
	In fact, $E/F$ is an abelian Galois extension \cite[Prop.\ 1.7]{momose}. Each character $\chi_\sigma$ can be regarded as character of $\Ga\Q$. Its kernel is thus an open subgroup $H_\sigma$ of $\Ga\Q$.
	
	\begin{definition}\label{def:K}
	       Let $H = \bigcap_{\sigma\in\Gamma}H_\sigma$, and let $K$ be the corresponding Galois extension of $\Q$.
	\end{definition}
	
	\begin{remark}\label{rem:character-non-trivial}
		Let $\varepsilon'$ be the Nebentypus character of $f$. Then the central character $\varepsilon$ of $\pi$ is $(\varepsilon')^2$. If $\varepsilon'$ is non-trivial, then, using the Petersson inner product, $(c, (\varepsilon')\ii)$ is always an inner twist, where $c$ denotes complex conjugation. In particular, the group $H$ includes the kernel of $\varepsilon'$ and hence of $\varepsilon$.
	\end{remark}
	
	Recall from \Cref{thm:image-contains-sp4} that $\LL'$ is the set of primes $\l\ge 5$ such that $\rho_\l|_{\Ql}$ is de Rham, and crystalline if $\l\nmid N$.
	
	\begin{lemma}\label{lem:small-trace}
		For all but finitely many primes $\l\in\LL'$, $\rho_\l|_K$ descends to a representation
		\[\rho_\l|_K\:\Gal(\Qb/K)\to \Gf(F\tensor\Ql)\]
		that is defined over $F\tensor\Ql$.
	\end{lemma}
	
	\begin{proof}
		By part $(ii)$ of \Cref{thm:image-contains-sp4}, the residual representation $\orho_\l|_K$ is irreducible for all but finitely many primes $\l\in\LL'$. Hence, by the proof of \Cref{cor:field-of-def}, it is enough to show that $\Tr\rho_\l(\Gal(\Qb/K))\sub F\tensor\Ql$. By the Chebotarev density theorem, it is enough to show that for all primes $p\nmid \l N$ that split completely in $K$, we have $\Tr\rho_\l(\Frob_p) \in F$. But for such primes, if $\sigma\in \Gal(E/F) = \Gamma$, we have $\Tr\rho_\l(\Frob_p) = a_p$ and $\sigma(a_p) = a_p\chi_\sigma(p) = a_p$. The result follows.
	\end{proof}
	
	\begin{definition}\label{def:final-L}
		Define $\LL\sub\LL'$ to be the set of primes in $\LL'$  to which \Cref{lem:small-trace} applies, namely such that $\rho_\l|_K$ is isomorphic to a representation taking values in $\Gf(F\tensor\Ql)$.
	\end{definition}
	
	In particular, by \Cref{thm:image-contains-sp4} and \Cref{lem:small-trace}, the set $\LL$ has density $1$, and contains all but finitely many primes if $k_2>2$.
	
	\begin{definition}\label{def:G}
	    Let $\G$ be the group scheme over $\Z$ such that, for each $\Z$-algebra $R$, we have
    	\[
    	\G(R) = \set{(g, \nu)\in \Gf(\O_F\tensor_\Z R)\times R\t : \simil(g) = \nu^{k_1 + k_2 - 3}}.    
    	\]
        We define the trace and characteristic polynomial of an element $(g, \nu)\in \G(R)$ to be the trace and characteristic polynomial of $g\in \Gf(R)$, viewed as an element of $\GL_4(R)$.
	\end{definition}
	
	By construction, if $\l\in\LL$, then $\rho_\l|_K$ can be viewed as a representation
	\[\wrho_{\l}\:\Gal(\Qb/K)\to\G(\Ql)\]
	such that the projection to $\Gf(\O_F\tensor_\Z\Ql)$ is $\rho_\l|_K$, and the projection to $\Ql\t$ is the cyclotomic character. Moreover, by \cite{serre-mod-p-lattices}*{Thm.\ 5.2.1}, we can conjugate this representation to take values in $\G(\Zl)$.
	
	We give a second interpretation of the number field $F$. For each prime $p\nmid N$, let $b_p\in\O_E$ denote the coefficient of $X^2$ in the characteristic polynomial of $\rho_\l(\Frob_p)$ for any prime $\l\ne p$.
	
	\begin{lemma}\label{lem:b_q}
		Let $F_0 = \Q(\{\frac{b_p}{\varepsilon(p) }:p\nmid N\}) $. Then $F=F_0$.
	\end{lemma}
	
	\begin{proof}	
		Fix a prime $\l\in \LL'$ such that $\rho_\l$ is defined over $E\tensor\Ql$. Now, $F_0$ is the field generated by $\Tr\std\rho_\l(\Frob_p)$, where $\std\rho_\l$ is the standard representation obtained by composing $\rho_\l$ with the maps
		\[\Gf\to \PGSp_4\xrightarrow{\sim}\SO_5\to\GL_5.\]
		Indeed, we have \[\wedge^2\rho_\l\tensor\simil\ii \simeq \std\rho_\l \+ 1,\]
		and 
		\[\Tr\std\rho_\l(\Frob_p) = \frac{b_p}{\simil\rho_\l(\Frob_p)} -1 = \frac{b_p}{p^{k_1 + k_2 -3}\varepsilon(p)} -1.\]
		
		By \Cref{thm:image-contains-sp4}, changing our choice of $\l$ if necessary, we can assume that the projection $\std\prepower{\sigma_0}\rho_\l$ is irreducible for all but finitely many primes $\l\in\LL'$. Indeed, for all but finitely many $\l\in\LL'$, the image of the residual representation $\prepower{\sigma_0}\orho_\l$ contains $\Sp_4(\Fl)$. Hence, the image of $\std\prepower{\sigma_0}\orho_\l$ contains $\SO_5(\Fl)$, which is an irreducible subgroup of $\GL_5(\Fl)$. Hence, $\std\prepower{\sigma_0}\rho_\l$ is irreducible.
		
		In general, if $\rho_1, \rho_2\:G\to \Gf(\Qlb)$ are representations, then the projective representations
		\[\Proj\rho_1,\ \Proj\rho_2\: G\to \PGSp_4(\Qlb)\]
		are isomorphic if and only if $\rho_1$ and $\rho_2$ are character twists of each other. Hence, $\sigma\in\Gamma$ if and only if the two projective representations \[\Proj\prepower{\sigma_0}\rho_\l,\ \Proj\sig\rho_\l\:\Ga\Q\to \PGSp_4(\Qlb)\] 
		are isomorphic. Via the exceptional isomorphism, it follows that $\sigma\in\Gamma$ if and only if 
		$\std\prepower{\sigma_0}\rho_\l$ and $\std\sig\rho_\l$	are isomorphic as $\SO_5$-valued representations. Following the argument \cite{Ramakrishnan-appendix}*{pp.\ 35-36}, we see that $\std\prepower{\sigma_0}\rho_\l$ and $\std\sig\rho_\l$ are isomorphic as $\SO_5$ representations if and only if they are isomorphic as $\GL_5$ representations. Indeed, both representations are irreducible, and the claim follows from the lemma on page $34$ of \cite{Ramakrishnan-appendix}.   
		
		By the Brauer--Nesbitt theorem, we find that $\sigma\in\Aut(E/\Q)$ is an element of $\Gamma=\Aut(E/F)$ if and only if $\std\prepower{\sigma_0}\rho_\l$ and $\std\sig\rho_\l$ have the same trace, which, by the Chebotarev density theorem,  is equivalent to having  $\sigma(\Tr\std\rho_\l(\Frob_p)) = \sigma_0(\Tr\std\rho_\l(\Frob_p))$ for all $p\nmid \l N$. Running the argument again with a different choice of $\l$, we see that $\sigma\in\Gamma$ if and only if $\sigma$ fixes $\frac{b_p}{\varepsilon(p)}$ for all $p\nmid N$. Thus $\sigma\in \Aut(E/\Q)$ fixes $F$ if and only if $\sigma$ fixes $F_0$, so $F = F_0$.
	\end{proof}
	
	\section{The image of Galois}\label{sec:image}
	
	In this section, we prove \Cref{thm:precise-image}. Let $f$ be a cuspidal Siegel modular eigenform of weights $(k_1, k_2)$, $k_1\ge k_2\ge 2$ and level $N$. Assume, as always, that $f$ is of type \textbf{(G)} in the notation of \cite{schmidt}, and that it is not CM, RM or a symmetric cube lift. Let $\varepsilon$ be the central character of the automorphic representation $\pi$ corresponding to $f$ and let $E=\Q(\{a_p : p\nmid N\},\varepsilon)$ be the coefficient field of $f$, and let $F, K$ be the number fields defined in \Cref{def:F} and \Cref{def:K}. Let $\LL$ be the set of primes defined in \Cref{def:final-L} and, for $\l\in\LL$, write
	\[\wrho_{\l}\:\Gal(\Qb/K)\to\G(\Zl)\]
	for the Galois representation defined just after \Cref{def:G}.
	
	\subsection{Open image}\label{sec:open-image}
	
	In this subsection, we prove part $(i)$ of \Cref{thm:precise-image}. Note that the composition of $\wrho_\l$ with the projection to $\Zl\t$ is the cyclotomic character, and hence is surjective. Therefore, to show that the image of $\wrho_\l$ is open in $\G(\Zl)$, it is equivalent to show that the image of $\rho_\l|_K$ is an open subgroup of 
	\[G_\l := \set{g\in \Gf(\O_F\tensor_\Z\Zl) : \simil(g)\in \Zl^{\times(k_1 + k_2 - 3)}}.\]
	
	Fix a prime $\l\in\LL$ and write $F\tensor\Ql = \prod_{\lambda\mid \l}F_\lambda$, where the product is over the primes $\lambda$ of $F$ above $\l$. Fix a prime $\lambda\mid\l$. Denote by $\rho_\lambda|_K$ the representation
	\[\rho_\lambda|_K\:\Gal(\Qb/K)\to \Gf(F_\lambda)\]
	obtained via the projection $F\tensor\Ql\to F_\lambda$, and let
	\[\Proj\rho_\lambda|_K\:\Gal(\Qb/K)\to\PGSp_4(F_\lambda)\]
	be the associated projective Galois representation. Let $H_\lambda\adj$ be the image of $\Proj\rho_\lambda|_K$.
	
	\begin{lemma}
		$H_\lambda\adj$ is Zariski dense in $\PGSp_4(F_{\lambda})$, where $\PGSp_4$ is viewed as an algebraic group over $F_{\lambda}$.
	\end{lemma}
	
	\begin{proof}
		By the assumption that $\l\in\LL$, $\rho_\lambda$ is irreducible, and hence $\rho_\lambda|_K$ is irreducible by \cite{weiss2018image}*{Lem.\ 5.9}. Hence, by \cite{weiss2018image}*{Cor.\ 5.6, Rem.\ 5.5}, the Zariski closure of $H_\lambda\adj$ is $\PGSp_4(F_\lambda)$.
	\end{proof}
	
	\begin{lemma}\label{lem:open-image}
		$H_\lambda\adj$ is an open subgroup of $\PGSp_4(F_{\lambda})$.
	\end{lemma}
	
	\begin{proof}
		We argue as in \cite[Prop.\ 3.16]{conti2016galois}. $\PGSp_4$ is an absolutely simple, connected adjoint group over $F_{\lambda}$ and the adjoint representation of $\PGSp_4$ is irreducible. Since $H\adj_\lambda$ is Zariski dense in $\PGSp_4(F_{\lambda})$, by \cite[Thm.\ 0.7]{Pink}, there is a model $H$ of $\PGSp_4$, defined over a closed subfield $L\sub F_{\lambda}$, such that $H\adj_\lambda$ is an open subgroup of $H(L)$. Moreover, by \cite[Prop.\ 0.6]{Pink}, the field $L$ is exactly the trace field of $H_\lambda\adj$, i.e.\ we have $L = F_{\lambda}$, and hence $H = \PGSp_4$. It follows that $H\adj_\lambda$ is an open subgroup of $\PGSp_4(F_{\lambda})$.
	\end{proof}
	
	\begin{lemma}\label{lem:image-open-lambda}
		The image of $\rho_\lambda|_K$ is conjugate to an open subgroup of 
		\[G_\lambda := \set{g\in \Gf(\O_{F_{\lambda}}) : \simil(g) \in \Zl^{\times(k_1 + k_2 -3)}}.\]
	\end{lemma}
	
	\begin{proof}
		Up to conjugation, we may assume that the image $H_\lambda$ of $\rho_\lambda|_K$ is contained in $G_\lambda$. Arguing again as in \cite[Prop.\ 3.16]{conti2016galois}, we see that $H_\lambda$ must contain an open subgroup of $\Sp_4({F_{\lambda}})$. Indeed, by \Cref{lem:open-image}, the projective image $H\adj_\lambda$ of $H_\lambda$ is an open subgroup of $\PGSp_4(F_{\lambda})$. Since the map $\Sp_4(F_{\lambda})\to \PGSp_4(F_{\lambda})$ has degree $2$, and since $H_\lambda\cap \Sp_4(F_{\lambda})$ surjects onto $H_\lambda\adj\cap \PGSp_4(F_{\lambda})$, $H_\lambda$ must contain an open subgroup of $\Sp_4({F_{\lambda}})$. In other words, $H_\lambda$ contains a principal congruence subgroup of $\Sp_4(\O_{F_{\lambda}})$. Thus, $H_\lambda\cap \Sp_4(\O_{F_{\lambda}})$ is open in $\Sp_4(\O_{F_{\lambda}})$. Since the similitude of $\rho_\lambda$ surjects onto $\Zl^{\times(k_1 + k_2 -3)}$, it follows that $H_\lambda$ is open in $G_\lambda$. 
	\end{proof}

	\begin{proof}[Proof of Theorem $\ref{thm:precise-image}~(i)$]
		We argue as in \cite{momose}*{Thm.\ 4.1}. We first show that, if $\lambda_1, \lambda_2$ are two distinct primes of $F$ above $\l$, then the representations $\rho_{\lambda_1}|_K$ and $\rho_{\lambda_2}|_K$ are not isomorphic when restricted to any finite extension. 
		
		For each $i$, let $\widetilde{\lambda}_i$ be a prime of $E$ above $\lambda_i$. Then, from the diagram
		\[\begin{tikzcd}
\Gal(\Qb/K) \arrow[r, "\rho_{\widetilde\lambda_{i}}|_K"] \arrow[rd, "\rho_{\lambda_{i}}|_K"'] & \Gf(E_{\widetilde\lambda_i})           \\
                                                                                          & \Gf(F_{\lambda_i}) \arrow[u, hook]
\end{tikzcd}\]
		it is clear that $\rho_{\lambda_i}|_K\simeq\rho_{\widetilde\lambda_i}|_K$ when	viewed as representations valued in $E_{\widetilde\lambda_i}$.
		
	The primes $\widetilde\lambda_1, \widetilde\lambda_2$ correspond to two embeddings $\sigma_{1}, \sigma_{2}\in \Hom(E, \Qlb)$. Moreover, since $\lambda_1\ne\lambda_2$, $\sigma_1|_F\ne\sigma_2|_F$. For each $i$, $\prepower{\sigma_i}\rho_\l\simeq\rho_{\widetilde\lambda_i}$. Suppose that $\prepower{\sigma_1}\rho_\l|_L\simeq\prepower{\sigma_2}\rho_\l|_L$ for some finite extension $L/\Q$. Since $\rho_{\lambda_1}$ and $\rho_{\lambda_2}$ remain irreducible upon restriction to any finite extension, it follows from the proof of Theorem 2(iii) on page 169 of \cite{raj} that $\prepower{\sigma_1}\rho_\l\tensor\chi\simeq\prepower{\sigma_2}\rho_\l$ for some Dirichlet character $\chi$. Hence, by the definition of $F$ as the field fixed by the inner twists, we must have $\sigma_1|_F=\sigma_2|_F$, a contradiction.
		
		We can now apply the analysis of \cite{ribet1976galois}*{Ch.\ IV, \S4}. Let $\overline \h = \h\tensor\Qlb$, where $\h$ is the Lie algebra of the image of $\rho_\l$. Let $\overline\g = \g\tensor\Qlb$, where $\g$ is the Lie algebra of $G_\l$. Then 
        \[\overline\g\sub \gsp_4(\O_F\tensor\Ql)\simeq \bigoplus_{\lambda} \gsp_4(F_\lambda).\]
        
        By \Cref{lem:image-open-lambda}, for each prime $\lambda$, the projection of $\overline\h$ to $\gsp_4(F_\lambda)$ is surjective. By the above argument, if $\lambda_1$ and $\lambda_2$ are distinct primes of $F$ above $\l$, then the projections of  $\overline\h$ and $\overline \g$ onto $\gsp_4(F_{\lambda_1})\times\gsp_4(F_{\lambda_2})$ are the same. By the Lie algebra version of Goursat's lemma \cite{ribet1976galois}*{Lemma, p.790}, it follows that $\overline\h = \overline\g$. Since $\h\sub\g$, it follows that $\h = \g$, i.e. that the image of $\rho_\l|_K$ is an open subgroup of $G_\l$. Hence, the image of $\wrho_\l$ is an open subgroup of $\G(\Zl)$.
	\end{proof}
	
	\subsection{The precise image for almost all primes}

	In this section, we prove part $(ii)$ of \Cref{thm:precise-image}. As in the previous section, to show that $\wrho_\l$ surjects onto $\G(\Zl)$, it is sufficient to show that $\rho_\l|_K$ surjects onto $G_\l$. Moreover, as the following lemma shows, it is sufficient to show that the residual representation $\orho_\l|_K$ surjects onto
	\[\overline G_\l := \set{g\in \Gf(\O_F\tensor_\Z\Fl): \simil(g)\in \Fl^{\times(k_1 + k_2 -3)}}.\]
	
	\begin{lemma}\label{lem:ribet-product}
		Let $\l \ge 5$ be prime, and let $F_1, \ldots, F_t$ be finite unramified extensions of $\Ql$ with rings of integers $\O_1, \ldots, \O_t$ and residue fields $\F_1, \ldots, \F_t$. Let $H$ be a closed subgroup of $\Sp_4(\O_1)\times\cdots \times \Sp_4(\O_t)$, which surjects onto $\PSp_4(\F_1)\times \cdots \times \PSp_4(\F_t)$. Then $H = \Sp_4(\O_1)\times\cdots \times \Sp_4(\O_t)$. 
	\end{lemma}
	
	\begin{proof}
		This is a generalisation of \cite{Ribet75}*{Thm.\ 2.1} and \cite{Loeffler-adelic}*{Lem.\ 1.1.1}. Since no proper subgroup of $\Sp_4(\F_j)$ surjects onto $\PSp_4(\F_j)$ for any $j = 1,\ldots, t$, it follows that $H$ surjects onto $\Sp_4(\F_1)\times \cdots \times \Sp_4(\F_t)$. The result now follows from \cite{DKR}*{Lem.\ 2}.
	\end{proof}
	
	\begin{corollary}\label{cor:gl}
	    Let $\l \ge 5$ be a prime that is unramified in $F$, and suppose that $H$ is a closed subgroup of $G_\l$ that surjects onto $\overline G_\l$. Then $H = G_\l$.
	\end{corollary}
	
	\begin{proof}
        Let $H'$ be the commutator subgroup of $H$ and let $H_0 = H\cap \Sp_4(\O_F\tensor_\Z\Zl)$. Clearly, $H_0 \supseteq H'$. Since $\Sp_4(\F)$ is a perfect group (i.e.\ a group that is equal to its own commutator subgroup) whenever $\F$ is a field of order at least $5$, we see that $H_0$ surjects onto $\Sp_4(\O_F\tensor_\Z\Fl)$. Therefore, by \Cref{lem:ribet-product}, $H_0 = \Sp_4(\O_F\tensor_\Z\Zl)$. Hence, $H$ is a closed subgroup of $G_\l$ that contains $\Sp_4(\O_F\tensor_\Z\Zl)$ and surjects onto $\Zl^{\times(k_1 + k_2 -3)}$. Thus, $H = G_\l$.
	\end{proof}

	Recall \Cref{lem:b_q}, that the field $F$ is equal to $\Q(\{\frac{b_p}{\varepsilon(p)} :p\nmid N\})$, where $b_p$ is the coefficient of $X^2$ in the characteristic polynomial of $\rho_\l(\Frob_p)$. 
	
	\begin{lemma}\label{lem:untwisted}
		There exists a prime $q\nmid N$ such that $F = \Q(b_q)$.
	\end{lemma}
	
	\begin{proof}
		We argue as in \cite[Thm.\ 3.1]{Ribet85}. Let $H_\l$ be the image of $\rho_\l|_K$ and, for an element $g\in H_\l$, let $b(g)$ denote the coefficient of  $X^2$ in its characteristic polynomial. Consider the set
		\[U = \set{g\in H_\l : \frac{b(g)}{\varepsilon(g)}\text{ generates }F\tensor\Ql \text{ as a }\Ql\text{-algebra}}.\]
		Then $U$ is an open subset of $H_\l$, which is in turn an open subgroup of $G_\l$ by part $(i)$ of \Cref{thm:precise-image}. Since $G_\l$ contains elements $g$ such that $\frac{b(g)}{\varepsilon(g)}$ generates $F\tensor\Ql$ as a $\Ql$-algebra, we see that $U$ is the intersection of two non-empty open subsets of $G_\l$, so is itself non-empty. Since $U$ is closed under conjugation, by the Chebotarev density theorem, there exists a rational prime $q\nmid\l N$ that splits completely in $K$ such that $\rho_\l(\Frob_q)\in U$. Hence, $\frac{b_q}{\varepsilon(q)}$ generates $F\tensor\Ql$ as a $\Ql$-algebra, so $F = \Q(\frac{b_q}{\varepsilon(q)})$. By \Cref{rem:character-non-trivial}, $\varepsilon(q) = 1$, so $F = \Q(b_q)$. 
	\end{proof}
	
	\begin{lemma}\label{lem:mod-lambda-image}
		For all but finitely many primes $\l\in\LL$, for all primes $\lambda\mid\l$ of $F$, the image of $\orho_\lambda|_K$ is exactly
		\[\overline G_\lambda :=\set{g\in\Gf(\F_\lambda ): \simil(g) \in \Fl^{\times(k_1 + k_2 -3)}}.\]
	\end{lemma}
	\begin{proof}
		We  follow a similar argument to \cite[Section 4.7]{Dieulefait2002maximalimages}. Since the similitude of $\orho_\lambda|_K$ surjects onto $\Fl^{\times(k_1 + k_2 -3)}$, it is enough to show that the image of $\orho_\lambda|_K$ contains $\Sp_4(\F_\lambda)$ for all $\lambda\mid\l$, for all but finitely many primes $\l\in\LL$. \
		
		Let $\l\in\LL$ and, for each prime $\lambda\mid\l$, let $\overline H_\lambda$ be the image of $\orho_\lambda|_K$ and let $\overline H_\lambda\adj\sub \PSp_4(\F_{\lambda})$ be the intersection of the projective image of $\orho_\lambda|_K$ with $\PSp_4(\F_\lambda)$. By \Cref{thm:image-contains-sp4}, we may assume that $\overline H_\lambda\adj$ contains $\PSp_4(\Fl)$. Now, suppose for contradiction that $\overline H_\lambda\adj\subsetneq \PSp_4(\F_\lambda)$. By Mitchell's classification of the maximal subgroups of $\PSp_4$ \cite{mitchell1914subgroups} (c.f.\ \cite{Dieulefait2002maximalimages}*{Sec.~3.1}), it follows that $\overline H_\lambda\adj\sub\PSp_4(\F)$ for some proper subfield $\F$ of $\F_\lambda$.
        
		Let $q$ be the prime from \Cref{lem:untwisted} such that $F = \Q(b_q)$. Suppose further that $\l\ne q$, that $\l$ does not divide the discriminant of the minimal polynomial of ${b_q}$ and that $a_q = \Tr\rho_\lambda(\Frob_q)$ is invertible modulo $\l$; these conditions exclude only finitely many primes $\l\in\LL$. Then $\F_\lambda = \Fl(\overline b_q)$, where $\overline b_q \equiv b_q\pmod \lambda$. Suppose that $\F\subsetneq \F_{\lambda}$. Then every element of $\overline H_\lambda\adj$ has a lift to $\overline H_\lambda$ with characteristic polynomial in $\F[X]$. Since the characteristic polynomial of $\orho_\lambda(\Frob_q)$ is $X^4 - a_qX^3 + b_qX^2 - a_qq^{k_1 + k_2 - 3}X + q^{2(k_1 + k_2 -3)}\pmod {\lambda}$, we see that there exists an element $t\in \F_{\lambda}\t$ such that $ta_q, t^2 b_q, t^3a_q\in \F$. Since $a_q$ is invertible modulo $\l$, it follows that $t^2 = \frac{t^3 a_q}{ta_q}\in \F\t$, and hence that $b_q\in\F$, contradicting the fact that $\F_{\lambda} = \Fl(\overline b_q)$.
		
		It follows that $\overline H_\lambda\adj=\PSp_4(\F_{\lambda})$ and hence that $\overline H_\lambda$ contains $\Sp_4(\F_{\lambda})$. Thus, $\overline H_\lambda = \overline G_\lambda$.
	\end{proof}
	
	\begin{remark}
	     A slightly weaker version of \Cref{lem:mod-lambda-image} also follows from combining \Cref{thm:image-contains-sp4} with \cite{de-reyna-dieulefait-wiese-1}*{Thm.\ 1.3}.
	\end{remark}
	
	Part $(ii)$ of \Cref{thm:precise-image} now follows inductively from Goursat's Lemma:
	
	\begin{lemma}[Goursat's Lemma, \cite{Ribet75}*{Lem.\ 3.2}]\label{lem:goursat}
		Let $G_1, G_2$ be groups and let $H$ be a subgroup of $G_1\times G_2$ for which the two projections $p_i\:H\to G_i$ are surjective. Let $N_1$ be the kernel of $p_2$ and let $N_2$ be the kernel of $p_1$. Then the image of $H$ in $G_1/N_1\times G_2/N_2$ is the graph of an isomorphism $G_1/N_1\xrightarrow{\sim} G_2/N_2$.
	\end{lemma}
	
	\begin{proof}[Proof of Theorem $\ref{thm:precise-image}~(ii)$]
		We argue as in the proof of $(3.1)$ of \cite{Ribet75}. If $\lambda_1, \lambda_2$ are two primes of $F$ above $\l$, then the image $\overline H$ of $\orho_{\lambda_1}|_K\times\orho_{\lambda_2}|_K$ is a subgroup of $\overline G_{\lambda_1}\times\overline G_{\lambda_2}$ and, by \Cref{lem:mod-lambda-image}, we can assume that the image of each of the two projections to $\overline G_{\lambda_i}$ is surjective. 
		
		Let $N_2$ and $N_1$ be the kernels of the projections of $\overline H$ onto $\overline G_{\lambda_1}$ and $\overline G_{\lambda_2}$. Then, by Goursat's lemma, the image of $\overline H$ in $\overline G_{\lambda_1}/N_1\times \overline G_{\lambda_2}/N_2$ is the graph of an isomorphism $\overline G_{\lambda_1}/N_1\xrightarrow{\sim}\overline G_{\lambda_2}/N_2$.
		
		Let $\overline G$ be the projection of $\overline G_\l$ onto $\overline G_{\lambda_1}\times\overline G_{\lambda_2}$. Since the kernel of $\overline G\to \overline{G}_{\lambda_1}$ is $\Sp_4(\F_{\lambda_2})$ we have $N_2\le \Sp_4(\F_{\lambda_2})$ and similarly $N_1 \le \Sp_4(\F_{\lambda_1})$. By the isomorphism $\overline G_{\lambda_1}/N_1\xrightarrow{\sim}\overline G_{\lambda_2}/N_2$ we have $N_2 =  \Sp_4(\F_{\lambda_2})$ if and only if $N_1 = \Sp_4(\F_{\lambda_1})$, in which case $\overline H = \overline G$.
		
		Thus, if $\overline H$ is a proper subgroup of $\overline G$, then, for each $i$, $N_i$ is a proper normal subgroup of $\Sp_4(\F_{\lambda_i})$, so $N_i\sub\{\pm I\}$.  The isomorphism ${\overline G_{\lambda_1}}/N_1 \xrightarrow{\sim}  {\overline G_{\lambda_2}}/N_2$ now implies that $\F_{\lambda_1} = \F_{\lambda_2}$ and that there are elements $\sigma\in \Gal(\F_{\lambda_1}/\Fl)$ and  $S \in \Gf(\F_{\lambda_1})$ such that, for each $(g_1, g_2)\in \overline G$, there is a scalar $\chi(g_1, g_2)$ such that $g_2 = \chi(g_1, g_2) \cdot\sigma (S g_1 S\ii)$. Since $\simil\orho_{\lambda_1}|_K=\simil\orho_{\lambda_2}|_K \in \Fl^{\times}$, it follows that $\chi(g_1, g_2)^2 = 1$ for all $(g_1, g_2)\in\overline G$. Hence, computing the characteristic polynomials of $g_2$ and $\chi(g_1, g_2) \cdot\sigma (S g_1 S\ii)$ and equating the coefficient of $X^2$, we find that
		\[b(g_2) = \sigma(b(g_1))\]
		for all $g_1, g_2\in \overline G$, where $b(g_i)$ denotes the coefficient of $X^2$ in the characteristic polynomial of $g_i$. 
		This contradicts \Cref{lem:untwisted}, which states there exists a prime $q$ such that $F = \Q(b_q)$, i.e.\ if $\l$ is large enough, the element $\Tr\std\orho_\l(\Frob_q)\in \O_F\tensor_\Z\Fl$ generates $\O_F\tensor_\Z\Fl$ as an $\Fl$-algebra.
		
		It follows that $\overline H$ of $\orho_{\lambda_1}|_K\times\orho_{\lambda_2}|_K$ surjects onto $\overline G$. Hence, by induction, the image of $\orho_{\l}|_K$ is $\overline G_\l$, and the result follows from \Cref{cor:gl}.
	\end{proof}
	
	\subsection{Adelic large image}
	
	Finally, we prove part $(iii)$ of \Cref{thm:precise-image}. We begin with some group theoretic results, which are mostly generalisations of \cite{serre-abelian}*{Ch.\ IV} and \cite{Loeffler-adelic}*{Sec.\ 1}.
	
	\begin{definition}[\cite{serre-abelian}*{IV-25}]
		Let $Y$ be a profinite group, and let $\Sigma$ be a non-abelian finite simple group. We say that $\Sigma$ \emph{occurs in $Y$} if there exist closed subgroups $Y_1, Y_2$ of $Y$ such that $Y_1\lhd Y_2$ and $Y_2/Y_1\cong \Sigma$. We write $\Occ(Y)$ for the set of non-abelian finite simple groups occurring in $Y$.
	\end{definition}
	
	\begin{lemma}[\cite{serre-abelian}*{IV-25}]\label{lem:serre-occ}
		If $Y = \varprojlim Y_a$ and each $Y\to Y_a$ is surjective, then $\Occ(Y) = \bigcup \Occ(Y_a)$. If $Y$ is an extension of $Y'$ and $Y''$, then $\Occ(Y) = \Occ(Y')\cup \Occ(Y'')$.
	\end{lemma}
	
	\begin{lemma}\label{lem:occ-sp4}
		Let $L$ be a finite extension of $\Ql$ for some prime $\l$, with ring of integers $\O$, uniformiser $\varpi$ and residue field $\F$. We have $\Occ(\GSp_4(\O)) = \Occ(\Sp_4(\O)) = \Occ(\PSp_4(\F))$.
	\end{lemma}
	
	\begin{proof}
		We argue as in \cite{kani-appendix}*{Lem.\ 10}. Since $\Gf(\O)/\Sp_4(\O)$ is abelian, the first equality follows from \Cref{lem:serre-occ}. Similarly, $ \Occ(\PSp_4(\F)) =  \Occ(\Sp_4(\F))$. Since $\Sp_4(\O) = \varprojlim_n\Sp_4(\O/\varpi^n)$, by \Cref{lem:serre-occ} again, we have $\Occ(\Sp_4(\O))= \bigcup_n\Occ(\Sp_4(\O/\varpi^n))$. It remains to show that $\Occ(\Sp_4(\O/\varpi^n)) = \Occ(\Sp_4(\F))$ for each $n$.
		
		Observe that the kernel $X$ of 
		\[\Sp_4(\O/\varpi^n)\to \Sp_4(\F)\]
		is an $\l$-group. Indeed, any matrix in the kernel can be written as $I + \varpi A$ where $A\in \M_4(\O/\varpi^n)$. Hence, in $\M_4(\O/\varpi^n)$, we have $(I+ \varpi A)^{\l^n} = I$, so every element of $X$ has order a power of $\l$.
		
		Since every $\l$-group is solvable, it follows that $\Occ(X) = \emptyset$, so $\Occ(\Sp_4(\O/\varpi^n)) = \Occ(\Sp_4(\F))$ by \Cref{lem:serre-occ}.
	\end{proof}
	
	\begin{corollary}\label{cor:occ}
		If $p\ne 2, \l$ and $q = p^r$ for some $r$, then $\PSp_4(\F_q)\notin \Occ(\Sp_4(\O))$.
	\end{corollary}
	
	\begin{proof}
		By \Cref{lem:occ-sp4}, $\Occ(\Gf(\O)) = \Occ(\PSp_4(\F))$. The result is now immediate from the classification of the maximal subgroups of $\PSp_4(\F)$ \cite{mitchell1914subgroups}.
	\end{proof}
	
	Recall that $\G$ is the group scheme whose $R$ points are
	\[\G(R) = \set{(g, \nu)\in \GSp_4(\O_F\tensor_\Z R)\times R\t : \simil(g) = \nu^{k_1 + k_2 -3}}.\]
	
	Let
		\[\G^\circ(R) = \Sp_4(\O_F\tensor_\Z R).\]
	
	In particular, we have $\G^\circ(\Zl) = \prod_{\lambda\mid\l}\Sp_4(\O_{F_\lambda})$, where the product is over primes $\lambda$ of $F$ above $\l$, and $\O_{F_\lambda}$ is the ring of integers of the completion $F_\lambda$ of $F$.
	
	\begin{theorem}\label{thm:adelic-image-circ}
		Fix a set of primes $\LL$. Let $U^\circ$ be a closed compact subgroup of $\G^\circ(\widehat{\Q}_\LL)$ such that:
		\begin{itemize}
			\item for every prime $\l\in\LL$, the projection of $U^\circ$ to $\G^\circ(\Ql)$ is open in $\G^\circ(\Ql)$;
			\item for all but finitely many primes $\l\in\LL$, the projection of $U^\circ$ to $\G^\circ(\Ql)$ is $\G^\circ(\Zl)$. 
		\end{itemize}
		Then $U^\circ$ is open in $\G^\circ(\widehat{\Q}_\LL)$.
	\end{theorem}
	
	\begin{proof}
		This is a generalisation of \cite{serre-abelian}*{Main Lemma, IV-19} and \cite{Loeffler-adelic}*{Thm.\ 1.2.2}. Let $S\sub \LL$ be a finite set of primes containing $\{2, 3,5\}\cap\LL$ and the finitely many primes such that the projection $U^\circ$ to $\G^\circ(\Ql)$ is not $\G^\circ(\Zl)$.
		
		First, note that the projection $U^\circ_S$ of $U^\circ$ to $\prod_{\l\in S}\G^\circ(\Ql)$ is open. We argue as in \cite{serre-abelian}*{Lem.\ 4 pp. IV-24}. Replacing $U^\circ$ with a finite index subgroup, we can assume that, for each $\l\in S$, the projection of $U^\circ$ to $\G^\circ(\Ql)$ is contained in the group of elements congruent to $1$ mod $\l$, i.e.\ is pro-$\l$. It follows that $U^\circ_S$ is pro-nilpotent, and hence is a product of its Sylow subgroups. Thus, $U^\circ_S\cong \prod_{\l\in S}U^\circ_\l$, where $U^\circ_\l$ is the projection of $U^\circ$ to $\G^\circ(\Ql)$. Since each $U^\circ_\l$ is open in $\G^\circ(\Ql)$ by assumption, it follows that $U^\circ_S$ is open in  $\prod_{\l\in S}\G^\circ(\Ql)$.
		
		To conclude, it is sufficient to show that $U^\circ$ contains an open subgroup of $\G^\circ(\widehat \Q_\LL)$. Since $U^\circ_S$ is open in  $\prod_{\l\in S}\G^\circ(\Ql)$, it is enough to show that $\prod_{\l\in\LL\setminus S}\G^\circ(\Zl)\sub U^\circ$. To show this, it is enough to show that $U^\circ$ contains $\G^\circ(\Zl) = \Sp_4(\O_F\tensor\Zl)$ for every $\l\in \LL\setminus S$. 
		
		For each $\l\in \LL\setminus S$, let $H_\l = U^\circ \cap \G^\circ(\Zl)$. By assumption, the projection of $U^\circ$ to $\G^\circ(\Ql)$ is $\G^\circ(\Zl)$, which in turn surjects onto $\PSp_4(\F_\lambda)$ for each $\lambda\mid\l$. Hence, $\PSp_4(\F_\lambda)\in \Occ(U^\circ)$. On the other hand, $U^\circ/H_\l$ is isomorphic to a closed subgroup of $\prod_{q\in \LL\setminus\{\l\}}\G^\circ(\Q_q)$, so by \Cref{cor:occ}, $\PSp_4(\F_\lambda)\notin\Occ(U^\circ/H_\l)$. It follows from \Cref{lem:serre-occ} that $\PSp_4(\F_\lambda)\in \Occ(H_\l)$. Hence, $H_\l$ is a subgroup of $\prod_{\lambda\mid\l}\Sp_4(\O_{F_\lambda})$ whose projection to $\prod_{\lambda\mid\l}\PSp_4(\F_\lambda)$ is surjective. It follows from \Cref{lem:ribet-product} that $H_\l = \prod_{\lambda\mid\l}\Sp_4(\O_{F_\lambda}) = \G^\circ(\Zl)$. The result follows.
	\end{proof}
	
	\begin{theorem}\label{thm:adelic-image}
		Fix a set of primes $\LL$. Let $U$ be a closed compact subgroup of $\G(\widehat{\Q}_\LL)$ such that:
		\begin{itemize}
			\item for every prime $\l\in\LL$, the projection of $U$ to $\G(\Ql)$ is open in $\G(\Ql)$;
			\item for all but finitely many primes $\l\in\LL$, the projection of $U$ to $\G(\Ql)$ is $\G(\Zl)$;
			\item the image of $U$ in $\widehat \Q_\LL\t$ is open.
		\end{itemize}
		Then $U$ is open in $\G(\widehat{\Q}_\LL)$. 
	\end{theorem}
	
	\begin{proof}
		We follow \cite{Loeffler-adelic}*{Thm.\ 1.2.3}. Let $U^\circ = U \cap \G^\circ(\widehat\Q_\LL)$. We claim that $U^\circ$ satisfies the hypotheses of \Cref{thm:adelic-image-circ}. Since $\G(\widehat\Q_\LL)/\G^\circ(\widehat\Q_\LL)\cong \widehat\Q_\LL\t$ is abelian, the group $U^\circ$ contains the closure of the commutator subgroup of $U$. When $\l\ge 3$, $\Sp_4(\O_F\tensor\Zl)$ is the closure of its commutator subgroup. Hence, if $\l\ge 3$ and $U$ surjects onto $\G(\Zl)$, then $U^\circ$ surjects onto $\G^\circ(\Zl) =\Sp_4(\O_F\tensor\Zl)$. When $\l=2$ the commutator subgroup of $\Sp_4(\O_F\tensor\Z_2)$ still has finite index. Hence, $U^\circ$ satisfies the hypotheses of \Cref{thm:adelic-image-circ}. Thus, $U$ contains an open subgroup of $\G^\circ(\widehat\Q_\LL)$. Since the image of $U$ in $\widehat \Q_\LL\t\cong \G(\widehat\Q_\LL)/\G^\circ(\widehat\Q_\LL)$ is open, it follows that $U$ is open in $\G(\widehat{\Q}_\LL)$.
	\end{proof}
	
	\begin{proof}[Proof of Theorem $\ref{thm:precise-image}~(iii)$]
	 	The result now follows immediately from \Cref{thm:precise-image} parts $(i)$ and $(ii)$ and \Cref{thm:adelic-image}.
	\end{proof}
	
	\subsection{The image of \texorpdfstring{$\Ga\Q$}{Galois over Q}}
	
	In \Cref{thm:precise-image}, we computed the image of $\wrho_\l$ for all but finitely many primes $\l\in\LL$. Suppose that $\l\in\LL$ is such that the image of $\wrho_\l$ is surjective. In this section, we use methods of E.\ Papier (see \cite{Ribet85}*{Thm.\ 4.1}) to compute the image of
	\[\rho_\l\:\Ga\Q\to \Gf(\O_E\tensor_\Z\Zl).\]
	Suppose that $(\sigma, \chi_\sigma)$ is an inner twist, and let $\sigma(\rho_\l)$ denote the representation obtained by composing $\rho_\l$ with $\sigma\:\O_E\tensor_\Z\Zl\to \O_E\tensor_\Z\Zl$. Then the representations $\sigma(\rho_\l)$ and $\rho_\l\tensor\chi_\sigma$ have the same trace. Since both are semisimple, it follows that they are isomorphic, so there is a matrix $X\in \Gf(E\tensor_\Q\Ql)$ such that
	\[X\sigma(\rho_\l)X\ii = \rho_\l\tensor\chi_\sigma.\]
    By definition, $\chi_\sigma|_K$ is trivial and $\sigma(\rho_\l|_K) = \rho_\l|_K$. It follows that $X$ commutes with the image of $\rho_\l|_K$. Since, for example, the image of $\rho_\l|_K$ contains $\Sp_4(\O_F\tensor_\Z\Zl)$, we see that $X$ is a scalar matrix, so there is an equality of matrices
    \[\sigma(\rho_\l(\gamma)) = \rho_\l(\gamma)\chi_\sigma(\gamma)\]
	for all $\gamma \in \Ga\Q$ and for all inner twists $(\sigma, \chi_\sigma)$.
	
	Recall that the group structure on the inner twists is given by 
	\[(\sigma, \chi_\sigma)\cdot(\tau, \chi_\tau) = (\sigma\tau, \chi_\sigma\cdot \sigma(\chi_\tau)),\]
where we view $\sigma, \tau$ as elements of $\Gamma = \Gal(E/F)$.
    
	Hence, for each $\gamma\in\Ga\Q$, the map $\sigma\mapsto\chi_\sigma(\gamma)$ defines an element of $H^1(\Gal(E/F),E\t)$. By Hilbert's theorem $90$, this cohomology group is trivial, so we can choose $\alpha(\gamma)\in E\t$ such that
	\[\chi_\sigma(\gamma) = \frac{\sigma(\alpha(\gamma))}{\alpha(\gamma)}\]
	for all $\sigma\in\Gamma$. Moreover, we can choose the elements $\alpha(\gamma)$ to be independent of $\l$, and so that $\alpha(\gamma)$ only depends on the image of $\gamma$ in $\Gal(K/\Q)$. Thus, there are exactly $[K:\Q]$ numbers $\alpha(\gamma)$ and, when $\l$ is large enough, we will have $\alpha(\gamma)\in \O_E\tensor_\Z\Zl$ for all $\gamma$. We deduce the following generalisation of \cite{Ribet85}*{Thm.\ 4.1}:
	
	\begin{theorem}\label{thm:precise-Q-image}
	    For all but finitely many primes $\l\in\LL$, the image of $\rho_\l$ is generated by 
	    \[G_\l = \set{g\in\Gf(\O_F\tensor_\Z\Zl) : \simil(g)\in \Zl^{\times(k_1 + k_2 -3)}}\]
	    together with the finite set of matrices
	    \[\dmat{\alpha(\gamma)}{\alpha(\gamma)}{\varepsilon(\gamma)/\alpha(\gamma)}{\varepsilon(\gamma)/\alpha(\gamma)}\]
	    where $\varepsilon$ is the central character of the automorphic representation $\pi$ associated to $f$.
	\end{theorem}
	
	\begin{proof}
	    By definition, for all $\gamma\in \Ga\Q$  and all $\sigma\in \Gamma$, we have
	    \[\rho_\l(\gamma)\alpha(\gamma)\ii = \rho_\l(\gamma)\chi_\sigma(\gamma)\sigma(\alpha(\gamma))\ii = \sigma(\rho_\l(\gamma)\alpha(\gamma)\ii).\]
	    Hence, $\rho_\l(\gamma)\alpha(\gamma)\ii\in \Gf(\O_F\tensor_\Z\Zl)$. Taking similitudes, we see that $\varepsilon(\gamma)\alpha(\gamma)^{-2}\in(\O_F\tensor_\Z\Zl)\t$.
	    
	    Moreover, we have an equality
	    \[\rho_\l(\gamma) = \dmat{\alpha(\gamma)}{\alpha(\gamma)}{\varepsilon(\gamma)/\alpha(\gamma)}{\varepsilon(\gamma)/\alpha(\gamma)}\br{\dmat{1}{1}{\alpha(\gamma)^2/\varepsilon(\gamma)}{\alpha(\gamma)^2/\varepsilon(\gamma)} \alpha(\gamma)\ii\rho_\l(\gamma)},\]
	    and the product in the second set of brackets belongs to $G_\l$. The result now follows from \Cref{thm:precise-image}.
	\end{proof}
	
	\begin{corollary}\label{cor:precise-Q-image}
	    For all but finitely many primes $\l\in\LL$, the image of $\rho_\l$ is the disjoint union of at most $[K:\Q]$ cosets
	    \[\coprod\dmat{\alpha(\gamma)}{\alpha(\gamma)}{\varepsilon(\gamma)/\alpha(\gamma)}{\varepsilon(\gamma)/\alpha(\gamma)}G_\l,\]
	    where $\gamma$ ranges over some subset of $\Gal(K/\Q)$.
	\end{corollary}
	
	\section{The Chebotarev density theorem}\label{sec:cdt}
	
	Let $K$ be a number field and let $L/K$ be a Galois extension with Galois group $G$. Assume that $L$ is unramified outside a finite set of primes. Let $M_K$ denote the set of primes of $K$ and, for each prime $\p\in M_K$ that is unramified in $L$, let $\Frob_\p\in G$ be a choice of Frobenius element. Let $C$ be a non-empty subset of $G$ that is stable under conjugation. 
	
	\begin{definition}\label{def:pi_c}
	    For any $x> 0$, define
	    \[\pi_C(x, L/K) := \#\set{\p\in M_K, \p\text{ unramified in }L : N_{K/\Q}(\p) \le x, \Frob_\p\in  C}.\]
	\end{definition}
	The Chebotarev density theorem states that
	\[\pi_C(x, L/K) \sim \frac{|C|}{|G|}\pi(x).\]
	To obtain explicit bounds on the size of $\pi_C(x, L/K)$, we will require an effective version of the Chebotarev density theorem. 
	
	\subsection{Unconditional effective Chebotarev}
	
	The following theorem is unconditional:
	
	\begin{theorem}[\cite{lmo}*{Thm.\ 1.4}]\label{thm:unconditional-cdt}
	    Assume that $L/K$ is finite. There exist constants $c_1, c_2$ such that, if 
	    \[\log x > c_1 (\log |\disc(L/\Q)|)(\log\log |\disc(L/\Q)|)(\log\log\log |6\disc(L/\Q)|),\]
	    then 
	    \[\pi_C(x, L/K) \le c_2\frac{|C|}{|G|}\Li(x),\]
	    where $\Li(x)=\int_2^x \frac{dt}{\log t}$ is the logarithmic integral function. 
	\end{theorem}
	
	Using this theorem, Serre \cite{ser81} shows how to deduce upper bounds for $\pi_C(x, L/K)$ assuming that $G=\Gal(L/K)$ is a compact $\l$-adic Lie group. Indeed, let $G$ be a compact $\l$-adic Lie group of dimension $D$ and let $C\sub G$ be a non-empty closed subset that is stable under conjugation.
	
	\begin{definition}[\cite{ser81}*{Sec.\ 3}]
	    Let $C_n$ denote the image of $C$ in $G/\l^nG$. We say that $C$ has \emph{Minkowski dimension} $\dim_M(C) \le d$ if $|C_n| = O(\l^{nd})$ as $n\to \infty$.
	\end{definition}
	
	Now, if $s\in C$, then the centraliser $Z_G(s)$ of $s$ is a closed Lie subgroup of $G$, so has a well-defined dimension.
	
	\begin{theorem}[\cite{ser81}*{Thm.\ 12}]\label{thm:cdt-serre}
	    Suppose that $\dim_M(C) \le d$, and set
	    \[r_C = \Inf_{s\in C}\dim\frac{G}{Z_G(s)}.\]
	    Then, for any $\epsilon >0$,
	    \[\pi_C(x, L/K) \ll \frac{x}{\log(x)^{1 + \alpha - \epsilon}},\]
	    where $\alpha = (D-d)/(D-r_C/2)$.
	\end{theorem}
	
	\subsection{Chebotarev density under the Generalised Riemann Hypothesis}
	
	Assuming that the Generalised Riemann Hypothesis holds for $L$, there are stronger effective versions of the Chebotarev density theorem. The following version is due to Lagarias--Odlyzko \cite{lo}.
	
	\begin{theorem}[\cite{ser81}*{Thm.\ 4}]\label{thm:conditional-cdt}
	    Assume that the Dedekind zeta function $\zeta_L(s)$ 
		satisfies the Riemann Hypothesis. Then 
		\[\pi_C(x, L/K)= \frac{|C|}{|G|}\Li(x)+O\bigg(\frac{|C|}{|G|}x^{{1}/{2}}(\log |\disc(L/\Q)|+[L:\Q]\log x)\bigg).\]
	\end{theorem}
	
	Finally, if we further assume that $L/K$ is abelian, then Artin's holomorphy conjecture is known to hold for $L/K$, and we obtain the following stronger result due to Zywina, which saves a factor of $\log x$ from an analogous result in \cite{mms}:
	
	\begin{theorem}[\cite{zyw}*{Thm.\ 2.3}]\label{zmain}
	    Suppose that $L/K$ is finite and abelian. Assume that the Dedekind zeta function $\zeta_L(s)$ 
		satisfies the Riemann Hypothesis. Then
	    \[	\pi_C(x, L/K)\ll \frac{| C|}{| G|}\frac{x}{\log x}+| C|^{1/2}[K:\Q]\frac{x^{1/2}}{\log x}
			\log M(L/K),\]
		where
		\[M(L/K):=2[L:K]\disc(K/\Q)^{1/[K:\Q]}\prod_{p \in \mathcal P(L/K)}p.\]
	\end{theorem}
	
	Here, we define $\mathcal P(L/K)$ to be the set of rational primes $p$ that are divisible by some prime $\p$ of $K$ that ramifies in $L$. 
	
	In order to apply these theorems, it is helpful to have a bound on $\log|\disc(L/\Q)|$. We will use the following result:
	
	\begin{proposition}[\cite{ser81}*{Prop.\ 5}]\label{prop:hensel}
	    Let $L/K$ be a finite extension of number fields. Then
	    \[\log |N_{K/\Q}(\disc(L/K))| \le ([L:\Q] - [K:\Q])\sum_{p\in\mathcal{P}(L/K)}\log p + [L:\Q]\log [L:K].\]
	\end{proposition}
	
	In order to estimate the size of $\pi_C(x, L/K)$, it is often more convenient to study the following weighted version:
	
	\begin{definition}\label{def:weighted-set}
	    For any $x>0$, define
	    \[\widetilde\pi_{C}(x, L/K) := \sum_{\substack{\p\in M_K, m\ge 1\\N(\p^m)\le x}}\frac1m \delta_C(\Frob_\p^m),\]
	    where $\delta_C\:G\to \{0,1\}$ is such that $\delta_C(g) = 1$ if and only if $g\in C$.
	\end{definition}
	
	This weighted sum $\widetilde\pi_{C}(x, L/K)$ is a good approximation of $\pi_C(x, L/K)$:
	
	\begin{lemma}[\cite{zyw}*{Lem.\ 2.7}]\label{zywina1}
	    We have
	    \[	\widetilde \pi_C(x, L/K)= \pi_C(x, L/K)+O\br{[K:\Q]\br{\frac{x^{1/2}}{\log x}+\log M(L/K)}},\]
	    where $M(L/K)$ is the constant defined in Theorem $\ref{zmain}$.
	\end{lemma}
	
	We end this section by  recalling the following result of Zywina \cite[Lemma 2.6]{zyw}.
	\begin{lemma}\label{zywina2}
		\begin{enumerate}
			\item
			Let $H$ be a subgroup of $G$ and suppose that every element of $C$ is conjugate to some element of $H$. Then
			\[\widetilde \pi_C(x, L/K)\le \widetilde \pi_{C\cap H} (x, L/L^{H}).\]
			\item 
			Let $N$ be a normal subgroup of $G$ an suppose that $NC \subset C.$ Then
			\[\widetilde \pi_C(x, L/K)= \widetilde \pi_{\overline C} (x, L^N/K),\]
			where $\overline C$ is the image of $C$ in $G/N=\Gal(L^N/K)$. 
		\end{enumerate}
	\end{lemma}
	
\section{Unconditional \texorpdfstring{bounds on $\pi_f(x, a)$}{Lang--Trotter bounds}}\label{sec:proof-unconditional}

Recall that $f$ is a cuspidal Siegel modular eigenform of weights $(k_1, k_2)$, $k_1\ge k_2\ge 2$ and level $N$. Assume that $f$ is of type \textbf{(G)}, and is not CM, RM or a symmetric cube lift. Let $\pi$ be the automorphic representation of $\Gf(\AQ)$ associated to $f$, and let $\varepsilon$ be its central character. Let $E=\Q(\{a_p : p\nmid N\},\varepsilon)$ be the coefficient field of $f$, and let $F, K$ be the number fields defined in \Cref{def:F} and \Cref{def:K}. Let $\LL$ be the set of primes defined in \Cref{def:final-L} and, for $\l\in\LL$, write
	\[\wrho_{\l}\:\Gal(\Qb/K)\to\G(\Zl)\]
	for the Galois representation defined just after \Cref{def:G}. Here,
	\[\G(\Zl) = \set{(g, \nu) \in \Gf(\O_F\tensor_\Z\Zl)\times\Zl\t : \simil(g) = \nu^{k_1 + k_2 -3}}.\]
	In particular, the projection of $\wrho_\l$ onto $\Gf(\O_F\tensor_\Z\Zl)$ is exactly $\rho_\l|_K$, and its projection to $\Zl\t$ is the $\l$-adic cyclotomic character. By \Cref{thm:precise-image}, $\wrho_\l$ has open image in $\G(\Zl)$ for all $\l\in\LL$ and is surjective for all but finitely many $\l\in\LL$. Using the fact that the dimension of the $\l$-adic Lie group $\Sp_4(\O_{F_\lambda})$ is $10[F_\lambda:\Ql]$ for each prime $\lambda\mid\l$, we see that
\begin{equation}\label{eq:dim-gl}
    \dim \G(\Zl) = {10[F:\Q] + 1}.
\end{equation}
    
\subsection{The case $a\ne 0$}

Fix a non-zero algebraic integer $a\in \O_F$. Our goal is to bound the size of the set $\pi_f(x, a) := \#\set{p\le x : a_p = a}$. 

    \begin{proposition}\label{prop:primes-of-K}
        Assume that $a\ne 0$. Then
        \[ \pi_f(x, a) = \frac{1}{[K:\Q]}\#\set{\p\in M_K, N(\p)=p\le x : a_p = a}.\]
    \end{proposition}
    
    \begin{proof}
        Recall from \Cref{sec:inner-twists} that, by definition, $\Gamma=\Gal(E/F)$ is the group of $\sigma\in \Aut(E/\Q)$ such that $(\sigma, \chi_\sigma)$ is an inner twist. Hence, if $a_p = a$ for some non-zero $a\in\O_F$, then, for every inner twist $(\sigma, \chi_\sigma)$, we have
        \[a_p = \sigma(a_p) =\chi_\sigma(p)a_p,\]
        from which it follows that $\chi_\sigma(p) = 1$. Since $K$ is, by definition, the field cut out by all the $\chi_\sigma$'s, we see that, if $a_p = a$ for a prime $p$, then $p$ splits completely in $K$. 
    \end{proof}
    
    It follows from \Cref{prop:primes-of-K} that bounding the size of $\pi_f(x, a)$ is exactly the same as bounding the size of $\#\set{\p\in M_K, N(\p)=p\le x : a_p = a}$, up to the constant $[K:\Q]$.
    
    Let
\[\CC_\l(a) = \set{(g,\nu)\in \im \wrho_\l \sub\G(\Zl): \tr(g) = a}.\]
Then, for any prime $\l\in\LL$, we have
\begin{align}\label{eq:comparison}
    \#\set{\p\in M_K, N(\p)=p\le x : a_p = a}  
    &= \pi_{\CC_\l(a)}(x, L/K) + O(1),
\end{align}
where $L$ is the fixed field of the kernel of $\wrho_\l$, and the $O(1)$ is to account for the finitely many primes $p\mid \l N$. In order to prove \Cref{thm:unconditional-lta}, we use \Cref{thm:cdt-serre} to estimate the size of $\pi_{\CC_\l(a)}(x, L/K)$.

\begin{proof}[Proof of Theorem $\ref{thm:unconditional-lta}~(i)$]
We show that the set $\CC_\l(a)$ has Minkowski dimension at most $9[F:\Q]+1$.  By \Cref{thm:precise-image}, $\im\wrho_\l$ is an open subgroup of $\G(\Zl)$. Hence, by $(\ref{eq:dim-gl})$, it has dimension $10[F:\Q]+ 1$ as an $\l$-adic Lie group. The set $\CC_\l(a)$ is the subvariety cut out by the equation $\tr(g) = a$. The equation $\tr(g) = a$ over $F\tensor_\Q{\Ql}$ can be viewed as $\dim_{\Ql}F\tensor_\Q{\Ql} = [F:\Q]$ equations over $\Ql$.
     By \cite{ser81}*{Sec.\ 3.2}, it follows that $\dim_M\CC_\l(a) \le (10[F:\Q] + 1) - [F:\Q] = 9[F:\Q] + 1$.
    
    It follows from \Cref{thm:cdt-serre} that 
    \[\pi_{\CC_\l(a)}(x, L/K) \ll \frac{x}{\log(x)^{1 + \alpha - \epsilon}},\]
    where $\alpha = \frac{[F:\Q]}{10[F:\Q]+1}$. The result follows from \Cref{prop:primes-of-K} and $(\ref{eq:comparison})$.
\end{proof}

\subsection{The case $a = 0$}

Let $PG_\l$ denote the image of the set
\[G_\l = \set{g\in\Gf(\O_F\tensor_\Z\Zl) : \simil(g)\in \Zl^{\times(k_1 + k_2 -3)}}\]
in $\PGSp_4(\O_E\tensor_\Z\Zl)$, and let $H_\l$ denote the image of $\Proj\rho_\l\:\Ga\Q\to \PGSp_4(\O_F\tensor_\Z\Zl)$. Let $L = \Qb^{H_\l}$ be the field cut out by the kernel of $\Proj\rho_\l$. By \Cref{thm:precise-Q-image}, there is a subgroup of $H_\l$ of index at most $[K:\Q]$ that is open in $PG_\l$. In particular, $H_\l$ and $PG_\l$ have the same dimension as $\l$-adic Lie groups, i.e.\ both have dimension $9[F:\Q] + 1$.

We can write $PG_\l = \prod_{\lambda\mid\l}PG_\lambda$, where $PG_\lambda$ is the image of the set
\[G_\lambda = \set{g\in\Gf(\O_{F_\lambda}) : \simil(g)\in \Zl^{\times(k_1 + k_2 -3)}}\]
in $\PGSp_4(\O_{F_\lambda})$. Since $H_\l$ is contained in a union of cosets of $G_\l$, we can similarly define $H_\lambda$ to be the projection of $H_\l$ onto the corresponding coset of $PG_\lambda$.

\begin{proof}[Proof of Theorem $\ref{thm:unconditional-lta}~(ii)$]
Set
\[\CC_\l(0) = \set{g\in H_\l : \tr(g) = 0}.\]
Note that having trace $0$ is well defined on $\PGSp_4(\O_E\tensor_\Z\Zl)$, so this definition makes sense.
As in the previous section, we have $\dim H_\l-\dim_M\CC_\l(0) \ge [F:\Q]$. 

    We compute the quantity
    \[r = r_{\CC_\l(0)} = \Inf_{s\in \CC_\l(0)}\dim \frac{H_\l}{Z_{H_\l}(s)}.\]
    If $s\in \CC_\l(0)$, then, by \Cref{cor:precise-Q-image}, we can write
    \[s= \dmat{\alpha}{\alpha}{\varepsilon/\alpha}{\varepsilon/\alpha}h,\]
    where $h\in PG_\l$ and $\alpha, \varepsilon\in \O_E\tensor\Zl$. Moreover, as in \Cref{thm:precise-Q-image}, we have $\varepsilon/\alpha^2\in \O_F\tensor\Zl$. Hence, rescaling by $\alpha\ii$, we can assume that $s\in \PGSp_4(\O_F\tensor\Zl)$. 
    
    Write $s = (s_\lambda)_\lambda$ via the isomorphism $F\tensor\Ql = \prod_{\lambda\mid\l}F_\lambda$. Then we have
    \[\dim \frac{H_\l}{Z_{H_\l}(s)} = \sum_{\lambda\mid\l}\dim H_\lambda - \dim Z_{H_\lambda}(s_\lambda).\]
    We now argue as in \cite{cojocaru-abelian-varieties}*{Thm.\ 1}. By \cite{cojocaru-abelian-varieties}*{Thm.\ A.1}, we have $\dim H_\lambda - \dim Z_{H_\lambda}(s_\lambda)\ge 4[F_\lambda:\Ql]$ for all $s\in \CC_\l(0)$. It follows that 
    \[\dim \frac{H_\l}{Z_{H_\l}(s)}\ge 4[F:\Q].\]
    Hence, by \Cref{thm:cdt-serre}, we have
    \[\pi_{\CC_\l(0)}(x, L/\Q) \ll \frac{x}{(\log x)^{1+\alpha - \epsilon}}, \]
    where 
    \[\alpha = \frac{[F:\Q]}{9[F:\Q] + 1 - 4[F:\Q]/2} = \frac{[F:\Q]}{7[F:\Q] + 1}.\]
    \end{proof}

\begin{remark}\label{rem:strengthened-result}
    Suppose that $a\in\O_F$ is such that there exists a prime $\l\in \LL$ such that $\l\mid \frac a 4$. Then, arguing as in the above proof (and as in \cite{cojocaru-abelian-varieties}*{Thm.\ 1}), it follows that $r_{\CC_\l(a)}\ge 4[F:\Q]$. Applying \Cref{thm:cdt-serre}, it follows that 
    \[\#\set{p\le x: a_p = a}\ll \frac{x}{(\log x)^{1+\alpha - \epsilon}}, \]
    where 
    \[\alpha = \frac{[F:\Q]}{8[F:\Q] + 1},\]
    a stronger bound than \Cref{thm:unconditional-lta}. In particular, if $\LL$ is the set of all primes, which is conjecturally the case when $E = F$, then, as in \cite{cojocaru-abelian-varieties}*{Thm.\ 1}, we obtain this stronger bound whenever $a\ne \pm4$.
\end{remark}

\section{Conditional \texorpdfstring{bounds on $\pi_f(x, a)$}{Lang--Trotter bounds}}\label{sec:proof-conditional}

Recall that $f$ is a cuspidal Siegel modular eigenform of weights $(k_1, k_2)$, $k_1\ge k_2\ge 2$ and level $N$. Assume that $f$ is of type \textbf{(G)}, and is not CM, RM or a symmetric cube lift. Let $\pi$ be the automorphic representation of $\Gf(\AQ)$ associated to $f$, and let $\varepsilon$ be its central character. Let $E=\Q(\{a_p : p\nmid N\},\varepsilon)$ be the coefficient field of $f$, and let $F, K$ be the number fields defined in \Cref{def:F} and \Cref{def:K}. Let $\LL$ be the set of primes defined in \Cref{def:final-L} and, for $\l\in\LL$, write
	\[\worho_{\l}\:\Gal(\Qb/K)\to\G(\Fl)\]
	for the reduction of $\wrho_\l$ modulo $\l$. Here,
	\[\G(\Fl) = \set{(g, \nu) \in \Gf(\O_F\tensor_\Z\Fl)\times\Fl\t : \simil(g) = \nu^{k_1 + k_2 -3}}.\]
	In particular, the projection of $\worho_\l$ onto $\Gf(\O_F\tensor_\Z\Fl)$ is exactly $\orho_\l|_K$, and its projection to $\Fl\t$ is the mod $\l$ cyclotomic character. By \Cref{thm:precise-image}, $\worho_\l$ is surjective for all but finitely many primes $\l\in\LL$.

    Recall that for $(g,\nu)\in \G(\Fl)$, we define $\tr(g,\nu)$ to be the trace of $g\in\Gf(\O_F\tensor_\Z\Fl)$, viewed as an element of $\O_F\tensor_\Z\Fl$. 
    
    \subsection{The case $a\ne 0$}
    
    Fix a non-zero algebraic integer $a\in \O_F$. Our goal is to bound the size of $\pi_f(x, a) = \#\set{p\le x : a_p = a}$. As in the previous section, by \Cref{prop:primes-of-K}, bounding the size of $\pi_f(x, a)$ is exactly the same as bounding the size of $\#\set{\p\in M_K, N(\p)=p\le x : a_p = a}$, up to the constant $[K:\Q]$. We will bound the size of this set by generalising the strategy of \cite{mms}.

    Let $p$ be a prime with $p\nmid N$ that splits completely in $K$, and let $\p$ be any prime of $K$ above $p$. Then the characteristic polynomial of $\wrho_\l(\Frob_\p)$ is equal to the characteristic polynomial of $\rho_\l(\Frob_p)$, which is a polynomial over $\O_F$ that is independent of $\l$. Define $F(p)$ to be its splitting field over $\Q$ and define
    \begin{equation*}
        \pi(x, a ; \l) := \#\set{\p\in M_K, N(\p) = p \le x: a_p = a, \ \l\text{ splits completely in }F(p)}.
    \end{equation*}
    
    The following lemma will allow us to use $\pi(x, a;\l)$ to bound $\pi_f(x,a)$:
    
    \begin{lemma}[c.f.\ \cite{mms}*{Lem.\ 4.4}]\label{lem:murty-bound}
	   Let $I$ be the interval $[y, y+u]$, where $y, u$ are chosen so that $x\ge y \ge u \ge y^{1/2}(\log y)^{1+ \epsilon}(\log xy)$ for some $\epsilon\ge 0$. Then, assuming GRH, we have
            \[\#\set{\p\in M_K, N(\p)=p\le x : a_p = a} \ll \max_{\l\in I}\pi(x, a; \l).\]
    \end{lemma}
    
    To prove \Cref{lem:murty-bound}, we will require the following bound on the discriminant of $F(p)$:
    
    \begin{lemma}\label{lem:log-disc}
     We have $\log|\disc(F(p)/\Q)| = O(\log p)$.
 \end{lemma}
 
 \begin{proof}
    Let $g(x)\in \O_F[x]$ be the characteristic polynomial of $\rho_\l(\Frob_p)$. Let $\pi$ be the unitary cuspidal automorphic representation of $\GSp_4(\AQ)$ associated to $f$. By assumption, $\pi$ lifts to a cuspidal automorphic representation $\Pi$ of $\GL_4(\AQ)$. Let $\alpha_1, \ldots, \alpha_4$ be the Satake parameters of the local representation $\Pi_p$. By \cite{jacquetshalika1}*{Cor.\ 2.5}, we have
    \[p^{-\frac12} < |\alpha_i| < p^{\frac12}\]
    for each $i$. In fact, if the weight $k_2 >2$, then the Ramanujan conjecture is known, and we have $|\alpha_i| = 1$.
    
    By definition, the four roots of $g(x)$ are $\alpha_1p^{\frac12(k_1 + k_2-3)}, \ldots, \alpha_4p^{\frac12(k_1 + k_2-3)}$. It follows that the coefficients of $g(x)$ are $O(p^{2(k_1 + k_2 -2)})$. Since the discriminant of $g(x)$ is a polynomial in its coefficients, we see that $\log|\disc(g(x))| = O(\log p)$.
    
    Let $L_i = F(\alpha_ip^{\frac12(k_1 + k_2-3)})$. Then 
    \begin{align*}
        \disc(L_i/\Q) &= N_{F/\Q}(\disc(L_i/F))\cdot \disc(F/\Q)^{[L_i:F]}.
    \end{align*}
    Since $\disc(F/\Q)^{[L_i:F]} = O(1)$, it follows from the above reasoning that $\log|\disc(L_i/\Q)| = O(\log p)$. But $F(p)$ is the compositum of the $L_i$'s and, in general, $\d(K_1\cdot K_2/\Q)\mid \d(K_1/\Q)\cdot \d(K_2/\Q)$ for arbitrary fields $K_1, K_2$, where $\d$ denotes the relative different ideal. It follows that $\log|\disc(F(p)/\Q)| = O(\log p)$.
    \end{proof}
    
    \begin{proof}[Proof of Lemma $\ref{lem:murty-bound}$]
        Observe that
        \begin{equation}\label{pix}
        \sum_{\l\in I}\pi(x,a;\l) = \sum_{\substack{\p\in M_K\\N(\mathfrak p) = p\le x\\a_p = a}}\#\set{\l\in I:\l\text{ splits completely in }F(p)}.     
        \end{equation}
           By taking the trivial conjugacy class $C=\{1\}$ of the Galois group of $F(p)$ over $\Q$ and applying \Cref{thm:conditional-cdt}, under GRH, the size of the set $\set{\l\le z : \l\text{ splits completely in } F(p)}$ is equal to
        \begin{align*}
        \frac{1}{[F(p) : \Q]}\pi(z) + O\br{ \frac{1}{[F(p) : \Q]}z^{1/2} \br{\log p+[F(p) : \Q] \log z} },
          \end{align*}
             for any sufficiently large real number $z$. Here, we have used \Cref{lem:log-disc}, that $\log| \disc(F(p)/\Q)|\ll \log p$.
        
        Since $u \ge y^{1/2}(\log y)^{2 + \epsilon}$, under GRH, we have
	    \[\pi(y+u) - \pi(y)\gg \frac{u}{\log u}.\]
	    It follows that
	    \[\#\set{\l\in I : \l\text{ splits completely in } F(p)} \gg \pi(y+u) - \pi(y),\]
	    where the implied constant in the above estimate is uniform in $p$. Using this estimate in $(\ref{pix})$ yields	   
	   \[ \sum_{\substack{\p\in M_K\\N(\mathfrak p) = p\le x\\a_p = a}}1
	    \ll \frac{1}{\pi(y +u) - \pi(y)}\sum_{\l\in I}\pi(x,a ; \l)\ll \max_{\l\in I}\pi(x,a; \l).\]
        \end{proof}
    
    \begin{remark}
            In \cite{murty-mod-forms-ii}*{pp.\ 304}, Murty proves a two-dimensional analogue of \Cref{lem:murty-bound} without assuming GRH. Murty's method makes essential use of the fact that, in the elliptic modular forms case, the field $F(p)$ is a quadratic extension of $F$. Hence, the quantity $\#\set{\l\in I:\l\text{ splits completely in }F(p)}$ can be estimated via a character sum. In our case, $F(p)$ need not even be an abelian extension of $F$, so Murty's method does not apply. It would be interesting to see if a version of \Cref{lem:murty-bound} can be proven without assuming GRH. By combining the methods of this section with the unconditional Chebotarev density theorem of \cite{thorner-zaman}, such a result would lead to an improved unconditional bound in \Cref{thm:unconditional-lta}.
    \end{remark}

We record the following lemma, for which we have been unable to find a reference. Recall that with our convention that $J = \br{\begin{smallmatrix}
        	0 & 0 & 0 & 1\\
        	0 & 0 & 1 & 0\\
        	0 & -1 & 0 & 0\\
        	-1 & 0 & 0 & 0
        \end{smallmatrix}}$, the upper-triangular matrices are a Borel subgroup of $
        \Gf$.

\begin{lemma}\label{lem:gsp4-upper-triangular}
Let $k$ be a field and let $g\in \Gf(k)$. Suppose that the characteristic polynomial of $g$ splits completely over $k$. Then $g$ is conjugate in $\Gf(k)$ to an upper-triangular matrix.
\end{lemma}

\begin{proof}
Let $V = k^4$, equipped with the symplectic form $\langle x,y\rangle = x^TJy$ for $x,y\in V$.

Since all Borel subgroups of $\Gf$ are conjugate, it is enough to show that $g$ is contained in a Borel subgroup. Equivalently, it is enough to show that $g$ stabilises a complete isotropic flag
\[
0 \subset V_1 \subset V_2 \subset V,
\]
where $\dim V_1=1$ and $\dim V_2=2$.

Since the characteristic polynomial of $g$ splits completely over $k$, $g$ has an eigenvector $v\in V$. Set $V_1 = \Span(v)$. Then $V_1$ is $g$-stable. Since $\langle v,v\rangle = 0$, this subspace is $g$-stable and isotropic.
Now let
\[
V_3 = V_1^\perp = \{w\in V : \langle w,v\rangle = 0\}.
\]
Then $\dim V_3 = 3$, and $V_3$ is $g$-stable. Indeed, if $w\in V_3$, then \[\langle gw, v\rangle = \simil(g)\langle w, g\ii v\rangle = \simil(g)\lambda\ii\langle w, v\rangle = 0,\]
where $\lambda$ is the eigenvalue of $v$. Hence $gw\in V_3$.

Since $V_1\subset V_3$ are both $g$-stable, $g$ induces a linear map on the quotient $V_3/V_1$. This quotient has dimension $2$, and its characteristic polynomial divides the characteristic polynomial of $g$, so it also splits completely over $k$. Therefore $g$ has an eigenvector in $w + V_1$ in $V_3/V_1$.

Set $V_2 = \Span(v, w)$. Then $V_2$ is $g$-stable. Moreover, $V_2$ is isotropic: since $w\in V_3=V_1^\perp$, we have $\langle w,v\rangle = 0$, 
and since the form is alternating, we have $\langle v,v\rangle = \langle w,w\rangle = 0$. Thus the restriction of $\langle\cdot,\cdot\rangle$ to $V_2$ is identically zero.

We have shown that $g$ stabilises the complete isotropic flag $0 \subset V_1 \subset V_2 \subset V$. The result follows. 
\end{proof}

  Let $\l\in \LL$ be a prime such that $\worho_\l$ is surjective, and let $L$ be the field cut out by the kernel of $\worho_\l$. Then $L$ is a finite Galois extension of $K$ with Galois group $\G(\Fl)$. 
  
  If $(g, \nu)\in \G(\Fl)$, then $g\in \Gf(\O_F\tensor_\Z\Fl)$, and we can define the characteristic polynomial of $g$ as a polynomial over $\O_F\tensor_\Z\Fl\simeq \prod_{\lambda\mid\l}\F_\lambda$. 
  
  Define:
\begin{align*}
    \CC(a, \l)&=\{(g,\nu) \in \G(\F_\l): \tr(g)= a\pmod\l, \text{ and all the eigenvalues of } g \text{ are in }\O_F\tensor_\Z\Fl \},\\
    	\mathcal B_{\l}& = \{(g, \nu)\in \G(\Fl) : g \text{ upper triangular}\},\\
    	\mathcal U_{ \l}&= \{(g, \nu)\in \G(\Fl) : g \text{ unipotent upper triangular}\},\\
    	\overline{\CC}(a, \l)&= \text{the image of } \CC(a, \l)\cap \mathcal B_{\l} \text{ in } \mathcal B_{\l}/ \mathcal U_{\l}.
\end{align*}
  	Then $\CC(a, \l)$ is a subset of $\G(\Fl)$ that is closed under conjugation.  Note that 
  	$\mathcal U_{\l}$ is normal in $\mathcal B_{\l}$ and that $\mathcal B_{\l}/\mathcal U_{\l}$ is abelian and is the Galois group $\Gal( L^{\mathcal U_{\l}}/ L^{\mathcal B_{\l}})$.
	
	\begin{lemma}\label{lem4}
	 Let $\l\in\LL$ be a rational prime that splits completely in $F$. Then
	    \[\pi(x,a; \l) \ll \widetilde\pi_{ \overline\CC(a, \l)}(x, L^{\mathcal U_\l}/L^{\mathcal B_\l}),\]
	where $\widetilde\pi$ is as in Definition $\ref{def:weighted-set}$.
	\end{lemma}
	\begin{proof}
        We first show that 
    \[
	    \pi(x, a;\l) \ll \pi_{\CC(a, \l)}(x, L/K).
	\]
        Let $\p\in M_K$ be a prime with norm $p\le x$ such that $a_p = a$ and $\l$ splits completely in $F(p)$. We need to show that $\worho_\l(\Frob_\p)\in \CC(a, \l)$.

        Since $\l$ splits in $F(p)$, all the roots of the characteristic polynomial of $\orho_\l(\Frob_p)$ are contained in $\Fl\t$. Since $\p$ has norm $p$, the eigenvalues of the $\Gf$ component of $\worho_\l(\Frob_{\p})$ are exactly the roots of the characteristic polynomial of $\orho_\l(\Frob_p)$. It follows that $\worho_\l(\Frob_\p)\in \CC(a, \l)$, and hence that $\pi(x, a;\l) \ll \pi_{\CC(a, \l)}(x, L/K)$.

	Now, $\CC(a, \l)$ is a union of conjugacy classes of $\G(\Fl)$. Moreover, if $(g, \nu)\in\CC(a, \l)$, then $g\in \Gf(\O_F\tensor_\Z\Fl) = \prod_{\lambda\mid \l}\Gf(\F_\lambda)$ has eigenvalues in $\O_F\tensor_\Z\Fl=\prod_{\lambda\mid \l}\F_\lambda$. It follows from \Cref{lem:gsp4-upper-triangular} that each projection $g_\lambda$ of $g$ is conjugate to an upper triangular matrix, and hence, by the Chinese remainder theorem, that $g$ is conjugate to an element of $\mathcal B_\l$. Hence, by \Cref{zywina2} $(i)$,
	    \[ \pi_{\CC(a, \l)}(x, L/K)\le \widetilde\pi_{\CC(a, \l)\cap \mathcal B_\l}(x, L/L^{\mathcal B_\l}).\]
	    Since multiplication by elements of $\mathcal U_\l$ preserves the set $\mathcal {B}_\l$, by \Cref{zywina2} $(ii)$, we have
	    \[\widetilde\pi_{\CC(a, \l)\cap \mathcal B_\l}(x, L/L^{\mathcal B_\l}) = \widetilde\pi_{\overline{\CC}(a, \l)}(x, L^{\mathcal U_\l}/L^{\mathcal B_\l}).\]
	    Combining the above estimates gives the desired result.
	\end{proof}

	 	\begin{lemma}\label{card1} Let $[F:\Q]=n$, and suppose that $\l$ is unramified in $F$. Then we have
	\begin{enumerate}[leftmargin=*]
	       \item $|\mathcal B_{\l}|\asymp \l^{6n+1}, |\mathcal U_{\l}|\asymp \l^{4n}$.
	     \item
	     $ |\overline{\CC}(a, \l)|\ll \l^{n+1}.$
	         \item
	      $[ L^{\mathcal B_{\l}}:K]\ll \l^{4n}$ and 
		$\log M( L^{\mathcal U_{\l}}/L^{\mathcal B_{\l}})\ll \log \l$, where $M(L^{\mathcal U_{\l}}/L^{\mathcal B_{\l}})$ is as in Theorem $\ref{zmain}$.
	\end{enumerate}
		\end{lemma}
	\begin{proof}
 If $\F$ is a finite field of cardinality $q$, then the set of upper triangular matrices in $\Gf(\F)$ is    
    \begin{equation*}\label{borel}
       \set{\br{\begin{smallmatrix}
			a &  &  & \\
			 & b &  & \\
			 &  & cb^{-1} & \\
			 &  & & ca^{-1}
		\end{smallmatrix}} \br{\begin{smallmatrix}
			1 & n &  & \\
			 & 1 &  & \\
			 &  & 1 &-n \\
			 &  & & 1
		\end{smallmatrix}} \br{\begin{smallmatrix}
			1 &  & r & s\\
			 & 1 & t & r\\
			 &  & 1 & \\
			 &  & & 1
		\end{smallmatrix}}: a,b,c \in \F^\times, n,r,s,t \in \F }.
    \end{equation*}
     Therefore, for any $\nu \in \F^\times$, it follows that
\begin{equation}\label{card_borel}
		    \#\set{g\in\Gf(\F):\simil(g) = \nu^{k_1 + k_2 -3},\ g\text{ upper triangular}}  = q^{6}+O(q^5)\asymp q^6
		\end{equation}
		and
        \begin{equation}\label{card_unipotent}
			\#\set{g\in\Gf(\F):g\text{ unipotent upper triangular}} = q^{4}+O(q^3)\asymp q^4.    
		\end{equation}
    \begin{enumerate}[leftmargin=*]
        \item Since $\l$ is unramified in $F$, we have $\O_F\tensor_\Z\Fl \simeq \prod_{\lambda\mid \l}\F_\lambda$, where the product runs over the primes $\lambda\mid \l$ of $F$.
    From $(\ref{card_borel})$, via this isomorphism, we have
			 \begin{align*}
                |\mathcal B_\l| &\asymp\sum_{\nu \in \Fl\t} \br{ \prod_{\lambda\mid\l}	\#\set{g\in\Gf(\F_\lambda):\simil(g) = \nu^{k_1 + k_2 -3},\ g\text{ upper triangular}} }\\
                &\asymp \sum_{\nu \in \Fl\t} \prod_{\lambda\mid\l}N(\lambda)^6 \asymp \l^{6n+1}.
            \end{align*}
            
            Similarly, from $(\ref{card_unipotent})$, we have
            \begin{align*}
                |\mathcal U_\l| &=\br{ \prod_{\lambda\mid\l}	\#\set{g\in\Gf(\F_\lambda):\ g\text{ unipotent upper triangular}} }\\
                &\asymp \prod_{\lambda\mid\l}N(\lambda)^4 \asymp \l^{4n}.
            \end{align*}
           
         \item From the definition of $ \overline\CC(a,\l)$, we observe that its  elements are in bijection with 
           \[
           \{(g, \nu)\in \G(\Fl) : g \text{ diagonal}, \tr(g)=a \}.
           \]
             Writing $a = (a_\lambda)_\lambda$ via the isomorphism $\O_{F}\tensor_\Z\Fl\simeq \prod_{\lambda\mid\l}\F_\lambda$ and proceeding as before, we obtain
            \begin{align*}
                |\overline\CC(a,\l)| 
                &\asymp\sum_{\nu \in \Fl\t} \br{\prod_{\lambda\mid\l}	\#\set{g\in\Gf(\F_\lambda):\simil(g) = \nu^{k_1 + k_2 -3},\ g\text{ diagonal},\ \tr(g) = a_\lambda}}\\
                &\ll\sum_{\nu\in\Fl\t}\prod_{\lambda\mid\l}N(\lambda) \asymp\l^{n+1}.
               \end{align*}
        \item Using formulae for the size of $\Gf(\F)$ over finite fields $\F$, it is easy to check that $|\G(\Fl)|\asymp\l^{10n+1}$. Since $[L^{\mathcal B_\l}:K] = [\G(\F_\l):\mathcal B_\l]$, it follows from $(i)$ that $[L^{\mathcal B_\l}:K]\ll \l^{4n}$. The second bound follows from part $(i)$, \Cref{prop:hensel} and the fact that $[L^{\mathcal U_\l}:L^{\mathcal B_\l}]= [\mathcal B_\l:\mathcal U_\l]$.
    \end{enumerate}
    \end{proof}

	\begin{proof}[Proof of Theorem $\ref{thm:conditional-lta} ~(i)$]
	    First observe that the group $\mathcal B_{\l}/\mathcal U_{\l}$ is abelian, which is the Galois group of the extension $ L^{\mathcal B_{\l}}/ L^{\mathcal U_{\l}}$. Hence, under GRH, applying \Cref{zmain}, \Cref{zywina1} and \Cref{card1} yields 
	\begin{align*}
		\widetilde{\pi}_{\overline\CC(a,\l)}(x, L^{\mathcal U_{\l}}/ L^{\mathcal B_{\l}})&\ll \frac{| \overline\CC(a,\l)| }{|{\mathcal B_{\l}}| /|{\mathcal U_{\l}}|}\frac{x}{\log x}+|\overline\CC(a,\l)|^{1/2}[ L^{\mathcal B_{\l}}:\Q]\frac{x^{1/2}}{\log x}\log M( L^{\mathcal U_{\l}}/ L^{\mathcal B_{\l}})\\
		& \ll \frac{1}{\l^n} \frac{x}{\log x}+\l^{\frac{1}{2}(9n+1)}\log \l \frac{x^{1/2}}{\log x}.
	\end{align*}
	Let $y\asymp\frac{x^{\alpha/n}}{(\log x)^{2\alpha/n}}$, where $\alpha = \frac{n}{11n+1}$. By \Cref{thm:precise-image}, the set of primes such that $\worho_\l$ is surjective has density $1$. Hence, for $y$ sufficiently large, we can choose $\l \in [y,2y]$ such that $\l$ splits completely in $F$ and such that $\worho_\l$ is surjective. By \Cref{lem4}, we have
	\begin{align*}
		\pi(x,a;\l) &\ll \frac{1}{y^n} \frac{x}{\log x}+y^{\frac{1}{2}(9n+1)}\log y\frac{x^{1/2}}{\log x}\\
		&\ll\frac{x^{1-\alpha}}{(\log x)^{1-2\alpha}}.
	\end{align*}
	The result now follows from \Cref{lem:murty-bound}.
	\end{proof}

\subsection{The case $a = 0$}

    For each prime $p$ with $p \nmid N$, let $F(p)$ be  the splitting field over $\Q$ of the characteristic polynomial of $\rho_{\l}(\Frob_p)$. Let 
    \begin{equation*}
        \pi(x, 0 ; \l) := \#\set{p \le x: a_p = 0, \ \l\text{ splits completely in }F(p)}.
    \end{equation*}

    Then the same proof as \Cref{lem:murty-bound} gives the following result:

        \begin{lemma}[c.f.\ \cite{mms}*{Lem.\ 4.4}]\label{lem:murty-bound2}
	   Let $I$ be the interval $[y, y+u]$, where $y, u$ are chosen so that $x\ge y \ge u \ge y^{1/2}(\log y)^{1+ \epsilon}(\log xy)$ for some $\epsilon\ge 0$. Then, assuming GRH, we have
            \[\pi_f(x, 0) \ll \#\set{p\le x : a_p = 0} \ll \max_{\l\in I}\pi(x, 0; \l).\]
    \end{lemma}

    Thus, we can use $\pi(x, 0; \l)$ to bound $\pi_f(x, 0)$.

    Let
    \[\overline G_\l := \set{g\in\Gf(\O_F\tensor_\Z\Fl) : \simil(g)\in \Fl^{\times(k_1 + k_2 -3)}}.\]
    Then, by \Cref{cor:precise-Q-image}, for all but finitely many primes $\l\in\LL$ the image of $\orho_\l\:\Ga\Q\to\Gf(\O_E\tensor_\Z\Fl)$ is a disjoint union of at most $[K:\Q]$ cosets
    \[\coprod\dmat{\alpha(\gamma)}{\alpha(\gamma)}{\varepsilon(\gamma)/\alpha(\gamma)}{\varepsilon(\gamma)/\alpha(\gamma)}\overline G_\l,\]
	    where $\gamma$ ranges over some subset of $\Gal(K/\Q)$. Moreover, taking $\l$ sufficiently large, the quantity $\varepsilon(\gamma)\alpha(\gamma)^{-2}$ is an element of $(\O_F\tensor_\Z\Fl)\t$. Write $\overline G_\l^\gamma$ for the coset indexed by $\gamma$.
    
     Fix a prime $\l\in\LL$ such that the image of $\orho_\l$ is as large as possible. Let $L$ be the field cut out by its kernel. Let
\begin{align*}
    \CC(0, \l)&=\{g \in \sqcup_\gamma \overline G_\l^\gamma: \tr(g)= 0\pmod\l, \text{ and all the eigenvalues of } g \text{ are in }(\O_E\tensor_\Z\Fl)\t \},\\
    	\mathcal B_{\l}& = \{g \in \sqcup_\gamma \overline G_\l^\gamma : g \text{ upper triangular}\},\\
    	\mathcal H_{ \l}&= \{g \in \sqcup_\gamma \overline G_\l^\gamma : g \text{ upper triangular with $4$ equal eigenvalues}\},\\
    	\overline{\CC}(0, \l)&= \text{the image of } \CC(0, \l)\cap \mathcal B_{\l} \text{ in } \mathcal B_{\l}/ \mathcal H_{\l}.
\end{align*}
    
    Since $[K:\Q]$ does not depend on $\l$, the proof of the following lemma is essentially identical to that of \Cref{card1}.

    \begin{lemma}\label{card2} Let $[F:\Q]=n$. Then we have
	\begin{enumerate}[leftmargin=*]
	       \item $|\mathcal B_{\l}|\asymp \l^{6n+1}, |\mathcal H_{\l}|\asymp \l^{5n}$.
	     \item
	     $ |\overline{\CC}(0, \l)|\ll \l.$
	         \item
	      $[ L^{\mathcal B_{\l}}:K]\ll \l^{4n}$ and 
		$\log M( L^{\mathcal H_{\l}}/L^{\mathcal B_{\l}})\ll \log \l$, where $M(L^{\mathcal H_{\l}}/L^{\mathcal B_{\l}})$ is as in Theorem $\ref{zmain}$.
	\end{enumerate}
		\end{lemma}

	\begin{proof}[Proof of Theorem $\ref{thm:conditional-lta}~(ii)$]
	$\CC(0, \l)$ is a union of conjugacy classes of $\bigsqcup_\gamma G_\l^\gamma$ and by \Cref{lem:gsp4-upper-triangular}, every element of $\CC(0, \l)$ is conjugate to an upper triangular matrix, i.e.\ an element of $\mathcal B_\l$. Hence, by \Cref{zywina2} $(i)$,
	    \[ \pi_{\CC(0, \l)}(x, L/K)\le \widetilde\pi_{\CC(0, \l)\cap \mathcal B_\l}(x, L/L^{\mathcal B_\l}).\]
	    Since multiplication by elements of $\mathcal H_\l$ preserves the set $\mathcal {B}_\l$, by \Cref{zywina2} $(ii)$, we have
	    \[\widetilde\pi_{\CC(0, \l)\cap \mathcal B_\l}(x, L/L^{\mathcal B_\l}) = \widetilde\pi_{\overline{\CC}(0, \l)}(x, L^{\mathcal H_\l}/L^{\mathcal B_\l}).\]

	Under GRH, by \Cref{zmain}, \Cref{zywina1} and \Cref{card2}, we have
	\begin{align*}
    \widetilde \pi_{\overline\CC(0,\l)} (x, L^{\mathcal H_{\l}}/L^{\mathcal B_{\l}})&\ll \frac{| \overline\CC(0,\l)| }{|{\mathcal B_{\l}}| /|{\mathcal H_{\l}}|}\frac{x}{\log x}+|\overline\CC(0,\l)|^{1/2}[ L^{\mathcal B_{\l}}:\Q]\frac{x^{1/2}}{\log x}\log M( L^{\mathcal H_{\l}}/ L^{\mathcal B_{\l}})\\
		& \ll \frac{1}{\l^{n}} \frac{x}{\log x}+\l^{\frac{1}{2}(8n+1)}\log \l \frac{x^{1/2}}{\log x}.
	\end{align*}
	Let $y\asymp\frac{x^{\alpha/n}}{(\log x)^{2\alpha/n}}$, where $\alpha = \frac{n}{10n+1}$. By \Cref{thm:precise-image}, the set of primes such that $\worho_\l$ is surjective has density $1$. Hence, for $y$ sufficiently large, we can choose $\l \in [y,2y]$ such that $\l$ that splits completely in $F$ and such that $\worho_\l$ is surjective. By the same argument as \Cref{lem4},
	\begin{align*}
		\pi(x,0;\l)
		&\ll\frac{x^{1-\alpha}}{(\log x)^{1-2\alpha}}.
	\end{align*}
	The result now follows from \Cref{lem:murty-bound2}.
	\end{proof}
	
	\section*{Acknowledgements}
	The authors are grateful to Tobias Berger, Andrea Conti, Tian Wang and Adri\'an Zenteno for helpful correspondences and comments. They also thank the anonymous referee for their helpful comments and corrections.
	The first author was supported by Grant No. 692854 from the European Research Council (ERC) and also by the ANRF  under the Start-Up Research Grant SRG/2023/001202. The second author was supported by Israeli Science Foundation grant 1400/19 and also by the ANRF  under the Start-Up Research Grant SRG/2023/000228. The third author was supported by BSF grant 2018250.
	
	\bibliography{bibliography}
	\bibliographystyle{alpha}
	\end{document}